\documentclass[article,reqno]{amsart}

\usepackage{amssymb,amsfonts,amsmath,amsthm}
\usepackage[all,arc]{xy}
\usepackage{enumerate}
\usepackage{mathrsfs}
\usepackage[toc,page]{appendix}
\usepackage[left=3cm, right=3cm, bottom=3cm]{geometry}
\usepackage{tabularx}
\usepackage{url}
\usepackage{color}
\usepackage{tikz-cd}  
\usepackage{mathdots} 
\usepackage{romannum} 
\usepackage{hyperref} 

\newtheorem{thm}{Theorem}[section]

\newtheorem*{thm*}{Theorem}
\newtheorem*{cor*}{Corollary}
\newtheorem*{prop*}{Proposition}
\newtheorem{cor}[thm]{Corollary}
\newtheorem{prop}[thm]{Proposition}
\newtheorem{lem}[thm]{Lemma}

\theoremstyle{definition}
\newtheorem{defn}[thm]{Definition}

\newtheorem*{notn*}{Notation}

\theoremstyle{remark}
\newtheorem{rem}[thm]{Remark}

\newtheorem*{idea*}{Idea}

\newcommand{\Spec}{{\rm Spec}}

\makeatletter
\let\c@equation\c@thm
\makeatother
\numberwithin{thm}{section}
\numberwithin{equation}{section}

\title[Moduli Spaces of Coherent Sheaves on Projective DM Stacks over Algebraic Spaces]{Moduli Spaces of Coherent Sheaves on Projective Deligne-Mumford Stacks over Algebraic Spaces}

\author{Hao Sun}

\begin{document}
\pagenumbering{arabic}
\maketitle
\begin{abstract}
In this paper, we study the geometric invariant theory on algebraic spaces, and construct the moduli space of $\mathcal{H}$-semistable sheaves on projective Deligne-Mumford stacks over algebraic spaces $S$. We prove that this moduli space is projective over $S$ as an algebraic space.
\end{abstract}

\flushbottom



\renewcommand{\thefootnote}{\fnsymbol{footnote}}
\footnotetext[1]{MSC2010 Class: 14A20, 14D20, 14D23}
\footnotetext[2]{Key words: moduli space, projective Deligne-Mumford stack, geometric invariant theory}

\section{Introduction}
Quot-functors on Deligne-Mumford stacks were studied by M. Olsson and J. Starr, which were proved to be representable by algebraic spaces \cite{OlSt}. Later on, F. Nironi constructed the moduli space of coherent sheaves on projective Deligne-Mumford stacks \cite{Nir}. As a special case of Deligne-Mumford stacks, the moduli space of locally free sheaves on orbifolds was understood as the moduli space parabolic bundles on its underlying space \cite{MehSes}. F. Nironi's approach gives a new construction of the moduli space of coherent sheaves on orbifolds with respect to a given generating sheaf. Based on Olsson-Starr-Nironi's work, people studied the moduli problem of framed bundles, Hitchin pairs and $\Lambda$-modules on (projective) Deligne-Mumford stacks, and constructed the corresponding moduli spaces \cite{BSP,Sun20194,Sun2020}.

In this paper, we construct the moduli spaces of $\mathcal{H}$-semistable coherent sheaves on Deligne-Mumford stacks $\mathcal{X}$ over algebraic spaces $S$. This construction is a generalization of M. Olsson, J. Starr \cite{OlSt} and F. Nironi's \cite{Nir} work.

Let $\mathcal{X}$ be a Deligne-Mumford stack over an algebraic space $S$. Denote by ${\rm \widetilde{Q}}(\mathcal{G},\mathcal{X})$ the quot-functor of coherent sheaves on $\mathcal{X}$, where $\mathcal{G}$ is a coherent sheaf on $\mathcal{X}$. M. Olsson and J. Starr proved that the quot-functor ${\rm \widetilde{Q}}(\mathcal{G},\mathcal{X})$ is represented by an algebraic space ${\rm Q}(\mathcal{G},\mathcal{X})$ \cite[Theorem 1.1]{OlSt}. Suppose that $S$ is an affine scheme or a noetherian scheme of finite type, and $\mathcal{X}$ is a projective Deligne-Mumford stack over $S$. The quot-space ${\rm Q}(\mathcal{G},\mathcal{X},P) \subseteq {\rm Q}(\mathcal{G},\mathcal{X})$ is a projective scheme over $S$ (see \cite[Theorem 1.5]{OlSt} and \cite[Theorem 2.17]{Nir}), where $P$ is an integer polynomial (as the modified Hilbert polynomial). To prove this result, the authors constructed a closed embedding
\begin{align*}
F_{\mathcal{E}}: {\rm Q}(\mathcal{G},\mathcal{X},P) \rightarrow {\rm Q}(F_{\mathcal{E}}(\mathcal{G}),X,P),
\end{align*}
where $\pi: \mathcal{X} \rightarrow X$ is the morphism to the coarse moduli space $X$ (as a projective scheme over $S$), $\mathcal{E}$ is a generating sheaf on $\mathcal{X}$ and $F_{\mathcal{E}}(\mathcal{F}):= \pi_*(\mathcal{F} \otimes \mathcal{E}^{\vee})$. Indeed, the existence of the generating sheaf plays an important role to construct this morphism \cite[Lemma 6.1 and Proposition 6.2]{OlSt}. This result can be generalized directly to the case that $\mathcal{X}$ is a projective Deligne-Mumford stack over an algebraic space $S$, and then ${\rm Q}(\mathcal{G},\mathcal{X})$ is an algebraic space projective over $S$ (see Proposition \ref{306}).

M. Olsson, J. Starr and F. Nironi's constructed the quot-functors and moduli spaces with respect to the following conditions:
\begin{itemize}
\item $S$ is an affine scheme or a noetherian scheme of finite type over an algebraically closed field $k$;
\item $\mathcal{X}$ is a projective Deligne-Mumford stack over $S$;
\item $\mathcal{H}$ is a generating sheaf.
\end{itemize}
In this paper, we generalize the above conditions to the following ones and construct the moduli space $\mathcal{H}$-semistable sheaves on $\mathcal{X}$.
\begin{itemize}
\item $S$ is an algebraic space of finite type over an algebraically closed field $k$;
\item $\mathcal{X}$ is a projective Deligne-Mumford stack over $S$ (see Definition \ref{303});
\item $\mathcal{H}$ is a locally free sheaf on $\mathcal{X}$.
\end{itemize}
Here are several problems we are faced when realizing this generalization.
\begin{enumerate}
\item Since $S$ is an algebraic space, it is easy to find that the quot-space ${\rm Q}(\mathcal{G},\mathcal{X},P)$ is an algebraic space (not necessarily to be a scheme). Therefore, we have to study the geometric invariant theory on algebraic spaces first.
\item With the same reason, we have to define the boundedness of families of sheaves over algebraic spaces.
\item Note that $\mathcal{H}$ is a locally free sheaf. Although the morphism
    \begin{align*}
        F_{\mathcal{H}}: {\rm Q}(\mathcal{G},\mathcal{X},P) \rightarrow {\rm Q}(F_{\mathcal{H}}(\mathcal{G}),X,P)
    \end{align*}
    still exists, this morphism may not be injective. Therefore, we have to find a subset ${\rm Q}^{\mathcal{H}}(\mathcal{G},\mathcal{X},P) \subseteq {\rm Q}(\mathcal{G},\mathcal{X},P)$ such that the restriction $F_{\mathcal{H}}|_{{\rm Q}^{\mathcal{H}}(\mathcal{G},\mathcal{X},P)}$ is a monomorphism.
\end{enumerate}
We construct the moduli space of $\mathcal{H}$-semistable sheaves on projective Deligne-Mumford stacks over algebraic spaces by solving the above problems.

As we discussed above, the first step is studying the geometric invariant theory on algebraic spaces. In \S 2, we define $G$-actions on algebraic spaces $X$. We consider an algebraic space $X$ as a groupoid $U_0/U_1$, and a $G$-action on $X$ is defined as a $G$-action on $U_0$ satisfying some compatible conditions (see Definition \ref{201} and \ref{202}). Denote by $(X,\sigma)$ an algebraic with a $G$-action $\sigma$. Then, we study the categorical quotient of a $G$-action $(X,\sigma)$ in the category of algebraic spaces. We find that a categorical quotient exists, if there is a categorical quotient of $(U,\sigma_u)$ (in the category of schemes) under some additional conditions.
\begin{thm}[Theorem \ref{206}]
Let $X$ be an algebraic space with a $G$-action $\sigma: G \times_S X \rightarrow X$. Let $(U,u)$ be a local chart of $X$, and denote by $\sigma_u$ the induced action on $U$. If there exists a categorical quotient $\phi: U \rightarrow V$ with respect to $\sigma_u$ satisfying the following conditions
\begin{enumerate}
\item $\phi$ is surjective;
\item the induced morphism $V_1 \rightarrow V \times_S V$ is injective;
\item the morphisms $(s,t): V_1 \rightarrow V_0$ are \'etale,
\end{enumerate}
where $V_1$ is the image of the following map
\begin{align*}
U \times_X U \hookrightarrow U \times_S U \xrightarrow{\phi \times_S \phi} V \times_S V,
\end{align*}
then there exists a (universal) categorical quotient of $(X,\sigma)$.
\end{thm}

After that, in \S 2.4, we study the geometric quotient of $(X,\sigma)$ in the category of algebraic spaces, and prove that a geometric quotient is also a categorical quotient (see Proposition \ref{211}). In \S 2.5, we define semistable and stable points on an algebraic space $X$ with respect to a given $G$-linearized invertible sheaf $L$ on $X$ (see Definition \ref{216}), and prove the following theorem, which is an analogue of \cite[Theorem 1.10]{MuFoKir}.
\begin{thm}[Theorem \ref{217}]
Let $X$ be a quasi-compact algebraic space over $S$ with a group action $\sigma: G \times_S X \rightarrow X$. Given a $G$-linearized invertible sheaf $L$ on $X$, there exists a uniform categorical quotient $(Y^{ss}(L),\phi)$ of $X^{ss}(L)$. Furthermore, there is an open subset $Y^{s}(L) \subseteq Y^{ss}(L)$ such that $X^{s}(L)=\phi^{-1}(Y^{s}(L))$ and $Y^s(L)$ is a uniform good quotient of $X^s(L)$.
\end{thm}
Finally, we study the Hilbert-Mumford criterion (see Proposition \ref{221}) on algebraic spaces as an application of Theorem \ref{217}. J. Heinloth generalizes the Hilbert-Mumford criterion to algebraic stacks \cite{Hein}, and what we study in this subsection (\S 2.6) is a special case of J. Heinloth's work. At the end of this section, we apply the Hilbert-Mumford criterion to Grassmannian ${\rm Grass}_S(r,n)$ over an algebraic space $S$ (see Corollary \ref{222}).

In \S 3, we construct the moduli space of $\mathcal{H}$-semistable coherent sheaves on projective Deligne-Mumford stacks $\mathcal{X}$ over algebraic spaces $S$, and the construction is based on the geometric invariant theory we studied in \S 2. Denote by $\pi: \mathcal{X} \rightarrow X$ the morphism to its coarse moduli space $X$, which is an algebraic space projective over $S$. When $\mathcal{H}$ is a generating sheaf, M. Olsson and J. Starr constructed the following morphism
\begin{align*}
F_{\mathcal{H}} : & {\rm Q}(\mathcal{G},\mathcal{X},P) \rightarrow {\rm Q}(F_{\mathcal{H}}(\mathcal{G}),X,P),\\
& [\mathcal{G} \rightarrow \mathcal{F}] \mapsto [F_{\mathcal{H}}(\mathcal{G}) \rightarrow F_{\mathcal{H}}(\mathcal{F}) ].
\end{align*}
This morphism is a finitely-presented closed embedding. However, when $\mathcal{H}$ is a locally free sheaf, this map may not be injective. Therefore, we want to take an open set ${\rm Q}^{\mathcal{H}}(\mathcal{G},\mathcal{X},P)$ of ${\rm Q}(\mathcal{G},\mathcal{X},P)$ such that the map $F_{\mathcal{H}}|_{{\rm Q}^{\mathcal{H}}(\mathcal{G},\mathcal{X},P)}$ is injective. To construct this open set, we would like to work on the universal family of ${\rm Q}(\mathcal{G},\mathcal{X},P)$. We show that the universal family exists in this case.
\begin{thm}[Theorem \ref{307}]
There exists a universal family over $X \times_S {\rm Q}(\mathcal{G},\mathcal{X},P)$.
\end{thm}
Usually, given a moduli problem $\widetilde{M}$, if $\widetilde{M}$ is represented by a scheme $M$, then the universal family corresponds to the identity map in $\widetilde{M}(M)={\rm Hom}(M,M)$. In our case,
\begin{align*}
{\rm \widetilde{Q}}(\mathcal{G},\mathcal{X}): ({\rm Sch}/S) \rightarrow {\rm Sets}
\end{align*}
is a functor from the big \'etale site of $S$-schemes to sets (more precisely, groupoids), and ${\rm Q}(\mathcal{G},\mathcal{X})$ is an algebraic space. It means that ${\rm \widetilde{Q}}(\mathcal{G},\mathcal{X}) ({\rm Q}(\mathcal{G},\mathcal{X}))$ is not well-defined. This is the reason why we have to construct the universal family in this case.

We define the following moduli problem
\begin{align*}
\widetilde{{\rm Q}}^{\mathcal{H}}(\mathcal{G},\mathcal{X},P): (\text{Sch}/S)^{{\rm op}} \rightarrow \text{Set}.
\end{align*}
For each $S$-scheme $T$, $\widetilde{{\rm Q}}^{\mathcal{H}}(\mathcal{G},\mathcal{X},P)(T) \subseteq \widetilde{{\rm Q}}(\mathcal{G},\mathcal{X},P)(T)$ consists of all quotients $[q: \mathcal{G}_T \rightarrow \mathcal{F}]$ such that the morphism $\theta_{\mathcal{H}}({\rm ker}(q))$ is surjective, where $\theta_{\mathcal{H}}$ is defined in Definition \ref{301}. We prove that $\widetilde{{\rm Q}}^{\mathcal{H}}(\mathcal{G},\mathcal{X},P)$ is representable by an open subset ${\rm Q}^{\mathcal{H}}(\mathcal{G},\mathcal{X},P) \subseteq {\rm Q}(\mathcal{G},\mathcal{X},P)$ (see Proposition \ref{308}). Furthermore, the restriction of the map $\mathcal{F}_{\mathcal{H}}|_{ {\rm Q}^{\mathcal{H}}(\mathcal{G},\mathcal{X},P)}$ is injective (see Proposition \ref{309}), and ${\rm Q}^{\mathcal{H}}(\mathcal{G},\mathcal{X},P) \rightarrow {\rm Q}(F_{\mathcal{H}}(\mathcal{G}),X,P)$ is a finitely-presented closed embedding (see Corollary \ref{310}).

In \S 3.3, we consider the boundedness of families of coherent sheaves over algebraic spaces. If $X \rightarrow S$ is a projective morphism of algebraic spaces, then a family $\mathscr{F}$ of coherent sheaves over $X \rightarrow S$ is equivalent to consider the corresponding family of sheaves on $X_U \rightarrow U$, where $U \rightarrow S$ is an \'etale surjective morphism and $X_U= X \times_S U$ (see Lemma \ref{312}). Therefore, the boundedness of families over algebraic spaces is equivalent to the boundedness of the corresponding families over schemes. Based on this fact, we show that the family of $\mathcal{H}$-semistable sheaves is bounded (see Corollary \ref{316} and \ref{317}).

In \S 3.4, we construct the moduli space of $\mathcal{H}$-semistable coherent sheaves with modified Hilbert polynomial $P$ on projective Deligne-Mumford stacks $\mathcal{X}$. In \S 3.3, we showed that the family of $\mathcal{H}$-semistable coherent sheaves is bounded. Then, there exists an integer $m$ such that $F_{\mathcal{H}}(\mathcal{F})$ is generated by global sections for any $\mathcal{H}$-semistable sheaves. There is an upper bound for the global sections of the family of $\mathcal{E}$-semistable sheaves on projective Deligne-Mumford stacks, where $\mathcal{E}$ is a generating sheaf (see \cite[Corollary 4.30]{Nir} and \cite[Corollary 3.13]{Sun2020}). Similarly, there is an upper bound for the family
\begin{align*}
\{h^0(\mathcal{X},\mathcal{F} \otimes \mathcal{H}^{\vee} \otimes \pi^* \mathcal{O}_X(m)) \text{ } | \text{ } & \mathcal{F} \text{ is a $\mathcal{H}$-semistable sheaf of dimension $d$} \\ & \text{ with modified Hilbert polynomial $P$}\}.
\end{align*}
Therefore, we can choose a positive integer $N$ large enough such that for any $\mathcal{H}$-semistable sheaf $\mathcal{F}$, we have
\begin{align*}
P(N) \geq P_{\mathcal{H}}(\mathcal{F},m)=h^0(X/S,F_{\mathcal{H}}(\mathcal{F})(m)).
\end{align*}
Denote by $V$ the linear space $S^{\oplus P(N)}$ and $\mathcal{G} \cong \mathcal{H} \otimes \pi^* \mathcal{O}_X(-N)$. We consider the quot-space ${\rm Q}(V \otimes \mathcal{G},\mathcal{X},P)$. We have the following morphism
\begin{align*}
{\rm Q}^{\mathcal{H}}(V \otimes \mathcal{G},\mathcal{X},P) \hookrightarrow {\rm Q}(F_{\mathcal{H}}(V \otimes \mathcal{G}),X,P),
\end{align*}
and if $N$ is large enough, we have the following embedding
\begin{align*}
    {\rm Q}( F_{\mathcal{H}}(V \otimes \mathcal{G}), X, P) \hookrightarrow {\rm Grass}(H^0(X/S,  F_{\mathcal{H}}(V \otimes \mathcal{G})(N) ),P(N)
    ).
\end{align*}
Therefore, there is a natural ${\rm SL}(V)$-action and a canonical invertible sheaf $\mathscr{L}_N$ on ${\rm Q}^{\mathcal{H}}(V \otimes \mathcal{G},\mathcal{X},P)$.

Denote by $Q^{\mathcal{H}}$ the subspace of ${\rm Q}^{\mathcal{H}}(V \otimes \mathcal{G},\mathcal{X},P)$ parametrizing quotients $[q: V \otimes \mathcal{G} \rightarrow \mathcal{F}]$ such that
\begin{enumerate}
\item the inducing morphism $\alpha: V \rightarrow H^0(X/S, F_{\mathcal{H}}(\mathcal{F})(N))$ is an isomorphism,
\item $\theta_{\mathcal{H}}( {\rm ker}(q) )$ is surjective.
\end{enumerate}
Let $Q_{ss}^{\mathcal{H}} \subseteq Q^{\mathcal{H}}$ be the subspace, if $[q: V \otimes \mathcal{G} \rightarrow \mathcal{F}]$ satisfies the following additional condition
\begin{enumerate}
\item[(3)] $\mathcal{F}$ is $\mathcal{H}$-semistable.
\end{enumerate}

We prove that if $[V \otimes \mathcal{G} \rightarrow \mathcal{F}] \in Q^{\mathcal{H}}$ is a semistable point (in the sense of GIT) if and only if $\mathcal{F}$ is a $\mathcal{H}$-semistable sheaf (see Theorem \ref{320}). This gives us the main theorem of the paper.

\begin{thm}[Theorem \ref{322}]
There exists a coarse moduli space $\mathcal{M}^{ss}(\mathcal{H},\mathcal{O}_X(1),P)$ of $S$-equivalence classes of $\mathcal{H}$-semistable sheaves with modified Hilbert polynomial $P$ on $\mathcal{X}$, and $\mathcal{M}^{ss}(\mathcal{H},\mathcal{O}_X(1),P)$ is an algebraic space projective over $S$.
\end{thm}

\section{Definition}
\subsection{Algebraic Spaces}
Let $S$ be a base scheme. An \emph{algebraic space over $S$} is a functor $X: (\text{Sch}/S)^{\rm op} \rightarrow {\rm Set}$ such that
\begin{enumerate}
\item $X$ is a sheaf with respect to the big \'etale topology;
\item the diagonal map $\Delta: X \rightarrow X \times_{S}X$ is representable by schemes;
\item  there is a surjective \'etale morphism $u:U \rightarrow X$, where $U$ is an $S$-scheme.
\end{enumerate}
An \emph{atlas} or a \emph{local chart} of $X$ is a pair $(U,u)$, where $U$ is an $S$-scheme and $u: U \rightarrow X$ is a surjective \'etale morphism. Let $(U,u)$ and $(V,v)$ be two local charts of $X$. A \emph{morphism of local charts} $(U,u)$ and $(V,v)$ is a morphism $f_{uv}:(U,u) \rightarrow (V,v)$ of schemes such that the following diagram commutes
\begin{center}
\begin{tikzcd}
U \arrow[rd, "u"] \arrow[rr, "f_{uv}"] &  & V \arrow[ld,"v"] \\
& X &
\end{tikzcd}
\end{center}

An algebraic space $X$ has property $P$ if there exists a local chart $(U,u)$ such that $U$ is a scheme with property $P$. A morphism $f : X \rightarrow Y$ of algebraic spaces is of property $P$ if the morphism $f$ is representable by schemes and there exists a local chart $(V,v)$ of $Y$ such that the morphism $V \times_X Y \rightarrow V$ has property $P$. For example the property $P$ could be imbedding, injection, surjection, properness. An algebraic space $X$ over $S$ is \emph{separated} if the diagonal map $\Delta: X \rightarrow X \times_S  X$ is a closed embedding.

Let $U_0$ be a $S$-scheme, and $U_1 \hookrightarrow U_0 \times_S U_0$ is a monomorphism, which defines an equivalence relation, such that the two natural projections $s,t: U_1 \rightarrow U_0$ are \'etale, where $s$ is the first projection known as the source map and $j$ is the second projection known as the target map. It is well-known that $U_0/U_1$ is an algebraic space. In fact, any algebraic space can be written in the form of sheaf quotients. Here is the construction. Let $X$ be an algebraic space, and we take a local chart $(U,u)$ of $X$. By the representability of the diagonal map, the product $U \times_{X} U$ is an $S$-scheme, and the two natural projections $s,t: U \times_X U \rightarrow U$ are \'etale. We have $X \cong U/U \times_X U$ as sheaves (see \cite[\S 5.2]{Ol} for more details). Taking two local charts $(U,u)$ and $(V,v)$ of $X$, let $f_{uv}:(U,u) \rightarrow (V,v)$ be a morphism. The following diagram
\begin{center}
\begin{tikzcd}
U_1 \arrow[r, hook] \arrow[d]  & U_0 \times_S U_0 \arrow[d] \\
V_1 \arrow[r, hook] & V_0 \times_S V_0
\end{tikzcd}
\end{center}
is cartesian. With respect to this property, we prefer to consider an algebraic space as sheaf quotients in this paper, and use the following notations
\begin{align*}
U_0:=U, \quad U_1:=U \times_X U, \quad X=U_0/U_1.
\end{align*}

Let $V_0/V_1$ and $W_0/W_1$ be algebraic spaces. A morphism $f_0:V_0 \rightarrow W_0$ induces a morphism $V_0 \times_S V_0 \rightarrow W_0 \times_S W_0$, but this morphism may not be well-defined when it restricts to $V_1$ and $W_1$. Therefore, a morphism $f_0: V_0 \rightarrow W_0$ induces a morphism $f: V_0/V_1 \rightarrow W_0/W_1$ of algebraic spaces if there exists a morphism $f_1: V_1 \rightarrow W_1$ such that
the diagram
\begin{equation}\tag{2.1}\label{eq201}
\begin{tikzcd}
V_1 \arrow[r, hook] \arrow[d, "f_1"]  & V_0 \times_S V_0 \arrow[d,"f_0 \times_S f_0"] \\
W_1 \arrow[r, hook]& W_0 \times_S W_0
\end{tikzcd}
\end{equation}
commutes. Note that if such a morphism $f_1$ exists, it is uniquely determined by $f_0$. On the other hand, given a morphism $f: Y (\cong V_0 /V_1) \rightarrow Z (\cong W_0 /W_1)$ of algebraic spaces, it can be naturally associated with a pair $(f_0,f_1)$ of morphisms. We want to remind the reader that this pair of morphisms is not unique, but it is determined uniquely up to isomorphism with respect to the local chart $(U,u)$ we choose.

A \emph{geometric point} of $X$ is a monomorphism $x: \Spec(k) \rightarrow X$, where $k$ is a field. We also use the notation $x \in X$ for a geometric point of $X$. Denote by $|X|$ the set of all geometric points in $X$, and there is a topology on $|X|$ induced by $X$ \cite[\S 6.3.3]{Ol}. We define the dimension $\dim_x(X)$ of $X$ at $x$ to be the dimension $\dim_u(U)$, where $U$ is any scheme admitting an \'etale surjection $U \rightarrow X$ and $u \in U$ is any point lying over $x$. We set
\begin{align*}
\dim(X):= \text{sup}_{x \in |X|} \dim_x(X).
\end{align*}

\subsection{$G$-action on Algebraic Spaces}
Let $G$ be a group scheme over $S$. To define a $G$-action on $X$, we first define a $G$-action on a local chart $(U,u)$ of $X$, and a $G$-action on $X$ is defined on local charts of $X$, which are compatible to each other. Here are the details of this construction.

Let $(U,u)$ be a local chart of $\mathcal{X}$. Let $\sigma_u: G \times_S U \rightarrow U$ be a $G$-action on the scheme $U$. Recall that the $G$-action $\sigma_u$ satisfies the following conditions
\begin{itemize}
    \item The diagram
\begin{center}
\begin{tikzcd}
G \times_S G \times_S U \arrow[r, "1_G \times \sigma_u"] \arrow[d, "\mu \times 1_U"]  & G \times_S U \arrow[d,"\sigma_u"] \\
G \times_S U \arrow[r, "\sigma_u"]& X
\end{tikzcd}
\end{center}
commutes, where $\mu : G \times_S G \rightarrow G$ is the multiplication.
\item The composition
\begin{align*}
U \cong S \times_S U \xrightarrow{e \times 1_U} G \times_S U \xrightarrow{\sigma_u} X,
\end{align*}
where $e: S \rightarrow G$ is the identity morphism for $G$.
\end{itemize}
The $G$-action $\sigma_u$ on $U$ induces a $G$-action $\sigma_{u \times_S u}$ on $U \times_S U$,
\begin{align*}
\sigma_{u \times_S u}: G \times_S (U \times_S U) \rightarrow (G \times_S U) \times_S (G \times_S U) \xrightarrow{\sigma_u \times \sigma_u} U \times_S U,
\end{align*}
where the first map is $(g,u_1,u_2) \rightarrow (g,u_1) \times (g,u_2)$. Note that the $G$-action $\sigma_{u \times_S u}$ may not induce a well-defined $G$-action on $U \times_X U$, and the problem is
\begin{align*}
\sigma_{u \times_S u}: G \times_S (U \times_X U) \nsubseteq U \times_X U.
\end{align*}
Therefore, if we want to define a $G$-action on $X = U_0/U_1$, the $G$-action on $U$ should also be well-defined on $U \times_X U$.

\begin{defn}\label{201}
A $G$-action on $(U,u)$ is a $G$-action $\sigma_u$ on $U$ such that the induced morphism $\sigma_{u \times_S u}: G \times_S (U \times_S U) \rightarrow U \times_S U$ is well-defined on $U \times_X U$, i.e.
\begin{align*}
\sigma_{u \times_S u}: G \times_S (U \times_X U) \subseteq U \times_X U.
\end{align*}
\end{defn}
We would like to use the notation $(U,u,\sigma_u)$ for a $G$-action on $(U,u)$, and if there is no ambiguity, we use the notation $\sigma_U$ for the group action $\sigma_u$ on $(U,u)$.

Let $(U,u)$ and $(V,v)$ be two local charts of $X$, and let $f_{uv}: (U,u) \rightarrow (V,v)$ be a morphism. Given $\sigma_u$ and $\sigma_v$ two group actions on local charts, we say that $\sigma_u$ and $\sigma_v$ are \emph{compatible}, if the following diagram commutes
\begin{center}
\begin{tikzcd}
G \times_S U \arrow[r, "1_G \times f_{uv}"] \arrow[d, "\sigma_u"]  & G \times_S V \arrow[d,"\sigma_v"] \\
U \arrow[r, "f_{uv}"]& V
\end{tikzcd}
\end{center}

\begin{defn}\label{202}
A $G$-action $\sigma$ on $X$ is given by the data and conditions
\begin{enumerate}
\item for each local chart $(U,u)$ of $X$, we have a $G$-action $(U,u,\sigma_u)$,
\item the $G$-actions on local charts are compatible.
\end{enumerate}
\end{defn}
We use the notation $\sigma: G \times_S X \rightarrow X$ for a $G$-action on $X$. Denote by  $\Phi$ the morphism
\begin{align*}
(\sigma,p_X): G \times_S X \rightarrow X \times_S X,
\end{align*}
where $p_X$ is the projection to the second factor $X$.

Let $x: \Spec(k) \rightarrow X$ be a geometric point. The \emph{orbit} of $x$ with respect to the group action $\sigma: G \times_S X \rightarrow X$ is the image of the following map
\begin{align*}
G \times_S \Spec(k) \xrightarrow{1_G \times x} G \times_S X \xrightarrow{\sigma} X.
\end{align*}
Denote by $\text{orb}(x)$ the orbit of $x$.

The \emph{stabilizer} of $x$ is the fiber product $\text{stab}(x)$ of the following diagram
\begin{center}
\begin{tikzcd}
\text{stab}(x) \arrow[r] \arrow[d]  & \Spec(k) \arrow[d,"x"] \\
G \times_S \Spec(k) \arrow[r, "(1_G \times x) \circ\sigma"]& X
\end{tikzcd}
\end{center}
For a geometric point $x \in X$, we have
\begin{align*}
\dim \text{orb}(x)+\dim \text{stab}(x)=\dim G.
\end{align*}
Let $U$ be a local chart of $X$. Denote by $u$ the point in $U$ lying over $x$. Since $U \rightarrow X$ is an \'etale covering, we have
\begin{align*}
\dim \text{orb}(x)= \dim \text{orb}(u), \quad \dim \text{stab}(x)= \dim \text{orb}(u).
\end{align*}
If the dimension of $\text{stab}(u)$ is constant in a neighborhood of $U$, we say $x$ is \emph{regular}. Denote by $X^{reg}$ the set of regular points in $X$. Denote by $S_r(X,\sigma)$ the set of points $x$ such that $\dim \text{stab}(x) \geq r$.

\begin{defn}\label{203}
A group action $\sigma: G \times_S X \rightarrow X$ is
\begin{enumerate}
    \item \emph{closed}, if for any geometric point $x \in X$, the orbit $\text{orb}(x)$ is closed,
    \item \emph{separated}, if the image of $\Phi: G \times_S X \rightarrow X \times_S X$ is closed and the image of $\Phi$ is $X \times_Y X$,
    \item \emph{proper}, if $\Phi$ is proper,
\end{enumerate}
\end{defn}

The definition of \emph{separatedness} is different from that of schemes (see \cite[Definition 0.8]{MuFoKir}).

\subsection{Categorical Quotient}
Let $X$ be an algebraic space over $S$ with a $G$-action $\sigma: G \times_S X \rightarrow X$. Denote by $(X,\sigma)$ an algebraic space with a $G$-action $\sigma$.
\begin{defn}\label{204}
A categorical quotient of $(X,\sigma)$ is a pair $(Y,\phi)$, where $Y$ is an algebraic space over $S$ and $\phi: X \rightarrow Y$ is a morphism, such that
\begin{enumerate}
    \item The diagram
    \begin{equation}\tag{2.2}\label{eq202}
    \begin{tikzcd}
        G \times_S X \arrow[r, "\sigma"] \arrow[d, "p_X"]  & X \arrow[d,"\phi"] \\
        X \arrow[r, "\phi"]& Y
    \end{tikzcd}
    \end{equation}
    commutes.
    \item Let $(Z,\psi)$ be a pair, where $Z$ is an algebraic space over $S$ and $\psi:X \rightarrow Z$ is a morphism, satisfying $\psi \circ \sigma=\psi \circ p_X$.
        \begin{equation}\tag{2.3}\label{eq203}
        \begin{tikzcd}
            & G \times_S X \arrow[r, "\sigma"] \arrow[d, "p_X"] & X \arrow[d, "\phi"] \arrow[ddr, bend left, "\psi"] & \\
            & X \arrow[r, "\phi"] \arrow[drr, bend right, "\psi"] & Y \arrow[dr, dotted, "\varphi" description] & \\
            & & & Z
        \end{tikzcd}
        \end{equation}
        Then, there is a unique morphism $\varphi: Y \rightarrow Z$ such that $\psi=\varphi \circ \phi$.
\end{enumerate}
\end{defn}
This definition of categorical quotients is given in the category of algebraic spaces, while the classical one is in the category of schemes (see \cite[Definition 0.5]{MuFoKir}).

Let $Y' \rightarrow Y$ be a morphism of algebraic spaces. Define $X'=X \times_Y Y'$ the fiber product. There is a natural $G$-action $\sigma'$ on $X'$ induced by $\sigma: G \times_S X \rightarrow X$, and denote by $\phi' : X' \rightarrow Y'$ the induced morphism.

\begin{defn}\label{205}
Let $\sigma$ be a $G$-action of $G$ on an algebraic space $X$. A pair $(Y,\phi)$ is called a \emph{universal} (resp. \emph{uniform}) categorical quotient of $(X,\sigma)$ if for all morphisms (resp. flat morphisms) $Y' \rightarrow Y$, the pair $(Y',\phi')$ is a categorical quotient of $X'$ with respect to the $G$-action $\sigma'$.
\end{defn}

Let $\sigma$ be a $G$-action on $X$. Given a local chart $(U,u)$, we have a $G$-action on $U$. Suppose that there exists a categorical quotient $\phi: U \rightarrow V$ under the action of the group $G$. We have the following map
\begin{align*}
U \times_X U \hookrightarrow U \times_S U \xrightarrow{\phi \times_S \phi} V \times_S V.
\end{align*}
Denote by $V_1$ the image of the above composition maps. Then, the map
\begin{align*}
(s,t): V_1 \hookrightarrow V_0 \times V_0
\end{align*}
is naturally induced by $U \times_X U \xrightarrow{(s,t)} U \times_S U$. Note that the morphisms $s,t: V_1 \rightarrow V_0$ are not necessarily to be \'etale.

If we want to construct the categorical quotient of $X=U_0/U_1$, we have to make the following assumptions:
\begin{itemize}
\item[$(\ast)$] there exists a categorical quotient (as schemes) of $\phi_0 : U_0 \rightarrow V_0$ such that
\begin{enumerate}
\item $\phi_0$ is surjective;
\item the induced morphism $V_1 \rightarrow V_0 \times_S V_0$ is injective;
\item the morphisms $(s,t): V_1 \rightarrow V_0$ are \'etale.
\end{enumerate}
\end{itemize}

\begin{thm}\label{206}
Let $X$ be an algebraic space with a $G$-action $\sigma: G \times_S X \rightarrow X$. Let $(U,u)$ be a local chart of $X$, and denote by $\sigma_u$ the induced action on $U$. If there exists a categorical quotient $\phi: U \rightarrow V$ with respect to $\sigma_u$ satisfying the condition $(\ast)$. Then, there exists a (universal) categorical quotient of $(X,\sigma)$.
\end{thm}

\begin{proof}
Set $U_0:=U$ and $U_1=U \times_X U$. We have $X \cong U_0/U_1$ as sheaves. By assumptions, denote by $(V_0,\phi_0)$ the categorical quotient of $U_0$ with respect to the $G$-action $\sigma_u$. We have the following natural map
\begin{align*}
U_1 \hookrightarrow U_0 \times_S U_0 \xrightarrow{\phi_0 \times_S \phi_0} V_0 \times_S V_0,
\end{align*}
and denote by $V_1$ the image of $U_1$ in $V_0 \times_S V_0$. Clearly, the following diagram
    \begin{center}
    \begin{tikzcd}
        U_1 \arrow[r, hook] \arrow[d]  & U_0 \times_S U_0 \arrow[d,"\phi_0 \times_S \phi_0"] \\
        V_1 \arrow[r, hook] & V_0 \times_S V_0
    \end{tikzcd}
    \end{center}
commutes. By the condition $(\ast)$, we know that the induced maps $s,t: V_1 \rightarrow V_0$ are \'etale and $V_1 \rightarrow V_0 \times_S V_0$ is an inclusion. The above discussion gives us an algebraic space $Y=V_0/V_1$ and a morphism $X \cong U_0 / U_1 \rightarrow Y=V_0 / V_1$. We will prove that the algebraic space $Y$ is a categorical quotient of $X$ with respect to the $G$-action $\sigma$. By the construction above, the first condition in Definition \ref{204} is satisfied, and we only have to show that $Y=V_0/V_1$ satisfies the second condition.

Let $(Z,\psi)$ be a pair, where $Z=W_0/W_1$ is an algebraic space and $\psi: X \rightarrow Z$ is a morphism, such that the diagram
\begin{equation}\tag{2.4}\label{eq204}
\begin{tikzcd}
    G \times_S X \arrow[r,"\sigma"] \arrow[d, "p_X"]  & X  \arrow[d,"\psi"] \\
    X \arrow[r,"\psi"] & Z
\end{tikzcd}
\end{equation}
commutes. The morphism $\psi: X \rightarrow Z$ can be considered as a pair $(\psi_0,\psi_1)$, where $\psi_0: U_0 \rightarrow W_0$ and $\psi_1: U_1 \rightarrow W_1$, such that
    \begin{equation*}
    \begin{tikzcd}
        U_1 \arrow[r, hook] \arrow[d,"\psi_1"]  & U_0 \times_S U_0 \arrow[d,"\psi_0 \times_S \psi_0"] \\
        W_1 \arrow[r, hook] & W_0 \times_S W_0
    \end{tikzcd}
    \end{equation*}
By \eqref{eq204}, we have
\begin{center}
\begin{tikzcd}
    G \times_S U_0 \arrow[r] \arrow[d]  & U_0  \arrow[d,"\psi_0"] \\
    U_0 \arrow[r,"\psi_0"] & W_0
\end{tikzcd}
\end{center}
We know that $V_0$ is the categorical quotient of $U_0$ with respect to the $G$-action. Therefore, there exists a unique morphism $\varphi_0:V_0 \rightarrow W_0$ making the following diagram
        \begin{center}
        \begin{tikzcd}
            & G \times_S U_0 \arrow[r] \arrow[d] & U_0 \arrow[d, "\phi_0"] \arrow[ddr, bend left, "\psi_0"] & \\
            & U_0 \arrow[r, "\phi_0"] \arrow[drr, bend right, "\psi_0"] & V_0 \arrow[dr, dotted, "\varphi_0" description] & \\
            & & & W_0
        \end{tikzcd}
        \end{center}
commutative.

We want to show that there exists a morphism $\varphi: Y \rightarrow Z$ satisfying the second condition in the definition of categorical quotients. The key step is to construct a morphism $\varphi_1: V_1 \rightarrow W_1$. Consider the diagram
    \begin{center}
    \begin{tikzcd}
            U_1 \arrow[r, "\phi_1"] \arrow[drr, bend right, "\psi_1"] & V_1 \arrow[dr, dotted, "\varphi_1" description] & \\
            & & W_1
    \end{tikzcd}
    \end{center}
and we define the morphism $\varphi_1$ in the following way
\begin{align*}
\varphi_1(v_1): =\psi_1(u_1), \text{ where } v_1=\phi_1(u_1).
\end{align*}
We have to check that this map $\varphi_1$ is well-defined. More precisely, let $u_{11}, u_{12} \in U_1$ be two elements such that $\phi_1(u_{11})=\phi_1(u_{12})=v_1$, and we will prove that $\psi_1(u_{11})=\psi_1(u_{12})$. We have the following diagram
        \begin{equation}\tag{2.5}\label{eq205}
        \begin{tikzcd}
            U_1 \arrow[r, "\phi_1"] \arrow[rr, bend left, "\psi_1"] \arrow[d, hook] & V_1 \arrow[r, dotted, "\varphi_1"] \arrow[d, hook] & W_1 \arrow[d, hook] \\
            U_0 \times_S U_0 \arrow[r, "\phi_0 \times \phi_0"] \arrow[rr, bend right, "\psi_0 \times \psi_0"] & V_0 \times_S V_0 \arrow[r, "\varphi_0 \times \varphi_0"]  & W_0 \times_S W_0
        \end{tikzcd}
        \end{equation}
The commutativity of the above diagram tells us that $\psi_1(u_{11})=\psi_1(u_{12})$. Therefore, the morphism $\varphi_1$ is well-defined, and it is easy to check that \eqref{eq205} is commutative with respect to the morphism $\varphi_1$. This finishes the proof of the existence of the morphism $\varphi: Y=V_0/V_1 \rightarrow Z=W_0/W_1$.

The above construction also implies that the morphism $\varphi_1$ is uniquely determined by $\varphi_0$. Therefore, the morphism $\varphi: Y \rightarrow Z$ is unique.

The case of universal categorical quotient can be proved similarly.
\end{proof}

\begin{rem}\label{207}
Let $(U,u)$ be an atlas of an algebraic space $X$. This lemma tells us that it is enough to work on the scheme $U$, and if a categorical quotient exists for $U$ under condition $(\ast)$, then it will give us a categorical quotient for the algebraic space $X$.

In the proof, we use the language of groupoid to prove this property of algebraic spaces. Indeed, we can define the categorical quotient of groupoids (in the category groupoids) similarly, and use the same approach to prove the existence of the categorical quotient of groupoids.
\end{rem}

\begin{lem}\label{208}
Let $\sigma$ be a $G$-action on an algebraic space $X$. Suppose that $(Y,\phi)$ is a categorical quotient of $X$ with respect to $\sigma$. If $G \times_S X \cong X \times_Y X$, then the map $\phi$ is representable by schemes.
\end{lem}

\begin{proof}
Let $u:U \rightarrow Y$ be a morphism, where $U$ is a scheme over $S$. We will prove that $X \times_Y U \cong G \times_S U$, which is a scheme. We have the following diagram
    \begin{center}
    \begin{tikzcd}
            G \times_S (X \times_Y U) \arrow[r] \arrow[d] & G \times_S X  \arrow[r] \arrow[d] & X \arrow[d,"\phi"] \\
            X \times_Y U \arrow[r] & X \arrow[r,"\phi"]  & Y  \\
            & X \times_Y U \arrow[u] \arrow[r] & U \arrow[u,"u"]
    \end{tikzcd}
    \end{center}
Note that $G \times_S (X \times_Y U) \cong (G \times_S U) \times_Y X$. Therefore,
    \begin{center}
    \begin{tikzcd}
            (G \times_S U) \times_Y X \arrow[r] \arrow[d] & X \arrow[d] \\
            U \times_Y X \arrow[r]   & Y
    \end{tikzcd}
    \end{center}
and we have
    \begin{center}
    \begin{tikzcd}
            (G \times_S U)  \arrow[r] \arrow[d] & X \arrow[d] \\
            U \arrow[r]   & Y
    \end{tikzcd}
    \end{center}
The above diagram implies that $G \times_S U \cong X \times_Y U$. Therefore, $X \times_Y U$ is a scheme.
\end{proof}

\begin{lem}\label{209}
Let $(Y,\phi)$ be a universal categorical quotient of $X$. Then, the morphism $\phi$ is surjective.
\end{lem}

\begin{proof}
Let $y=\Spec(k) \in Y$ be a point. We have the following cartesian diagram,
    \begin{center}
    \begin{tikzcd}
            X_y  \arrow[r] \arrow[d] & X \arrow[d] \\
            y \arrow[r]   & Y
    \end{tikzcd}
    \end{center}
and $y$ is a categorical quotient of $X_y$, which means that $X_y$ cannot be an empty set. Therefore, the morphism $\phi$ is surjective.
\end{proof}

\subsection{Geometric Quotient}

\begin{defn}\label{210}
Let $\sigma: G \times X \rightarrow X$ be a $G$-action on $X$. A pair $(Y,\phi)$, where $Y$ is an algebraic space and $\phi:  X \rightarrow Y$ is a morphism, is a \emph{geometric quotient} of $(X,\sigma)$ if it satisfies the following conditions
\begin{enumerate}
\item Diagram \eqref{eq202} commutes, i.e. $\phi \circ \sigma=\phi \circ p_X$.
\item The image of $\Phi: G \times_S X \rightarrow X \times_S X$ is $X \times_Y X$.
\item The morphism $\phi$ is surjective.
\item The morphism $\phi$ is submersive.
\item The structure sheaf $\mathcal{O}_Y$ is isomorphic to $(\phi_*(\mathcal{O}_X))^G$.
\end{enumerate}
A \emph{good quotient} of $(X,\sigma)$ is a geometric quotient $(Y,\phi)$ such that $Y$ is separated.
\end{defn}
Note that the properties of surjectivity and submersivity are defined for morphisms representable by schemes, which is implied by the second condition (see Lemma \ref{208}).

\begin{prop}\label{211}
A geometric quotient $(Y, \phi)$ of $(X,\sigma)$ is also a categorical quotient.
\end{prop}

\begin{proof}
We only have to prove that the geometric quotient $(Y,\phi)$ satisfies the universal property. Let $(Z,\psi)$ be a pair, where $Z$ is an algebraic space and $\psi: X \rightarrow Z$ is a morphism such that $\psi \circ \sigma=\psi \circ p_X$. We use the same notation as in the proof of Theorem \ref{206} that $X=U_0 / U_1$, $Y=V_0 /V_1$ and $Z=W_0 /W_1$, where $U_i$, $V_i$, $W_i$ are all schemes. By assumption that $(Y,\psi)$ is a geometric quotient of $(X,\sigma)$, the pair $(V_0,\psi_0)$ is also a geometric quotient of $(U_0,\sigma)$. By \cite[Proposition 0.1]{MuFoKir}, a geometric quotient is also a categorical quotient as schemes. Therefore, $(V_0,\psi_0)$ is a categorical quotient of $(U_0,\sigma_u)$. By Theorem \ref{206}, $(Y,\phi)$ is also a categorical quotient of $(X,\sigma)$ as algebraic spaces.
\end{proof}

\begin{lem}\label{212}
Let $X$ be an algebraic space with a $G$-action $\sigma: G \times_S X \rightarrow X$. If a geometric quotient $(Y,\phi)$ exists, then the action $\sigma$ is closed.
\end{lem}

\begin{proof}
Let $x: \Spec(k) \rightarrow X$ be a geometric point of $X$. Since $\phi:X \rightarrow Y$ is a submersive map, $\text{orb}(x)=\phi^{-1}(\phi(x))$ is a closed set. Therefore, the action $\sigma$ is closed.
\end{proof}

\begin{lem}\label{213}
Let $X$ be an algebraic space with a $G$-action $\sigma: G \times_S X \rightarrow X$. Suppose that a geometric quotient $(Y,\phi)$ of $(X,\sigma)$ exists. Then $Y$ is separated if and only if $\sigma$ is separated.
\end{lem}

\begin{proof}
Since $(Y,\phi)$ is a geometric quotient, the image of $\Psi:G \times_S X \rightarrow X \times_S X$ is $X \times_Y X$. On the other hand, we have the following Cartesian square
    \begin{center}
    \begin{tikzcd}
            X \times_Y X  \arrow[r] \arrow[d] & X \times_S X \arrow[d,"\phi \times \phi"] \\
            Y \arrow[r,"\Delta"]   & Y \times_S Y
    \end{tikzcd}
    \end{center}
where $\Delta: Y \rightarrow Y \times_S Y$ is the diagonal morphism. Since $\phi$ is submersive, $\Delta$ is a closed immersion if and only if $X \times_Y X$ is closed in $X \times_S X$. Therefore, the geometric quotient $Y$ is separated if and only if $\sigma$ is separated.
\end{proof}

\begin{prop}\label{214}
A universal categorical quotient $(Y,\phi)$ of $(X,\sigma)$ is a good quotient if and only if $\sigma$ is separated.
\end{prop}

\begin{proof}
If $(Y,\phi)$ is a good quotient, then $\phi$ is separated by Lemma \ref{213}.

On the other hand, if $(Y,\phi)$ is a universal categorical quotient, it satisfies conditions $(1)$ and $(4)$ automatically, and the morphism $\phi$ is surjective which is implied by the universal property. The separatedness of the action $\sigma$ implies
\begin{enumerate}
\item the map $\phi$ is submersive by Lemma \ref{212},
\item $Y$ is separated,
\item the image of $\Phi$ is $X \times_Y X$.
\end{enumerate}
Therefore, the $(Y,\phi)$ is also a geometric quotient.
\end{proof}

\begin{lem}\label{215}
Let $X$ be an algebraic space projective over $S$, and the space $X$ has a $\mathbb{G}_m$-action $\sigma$. The action $\sigma$ is proper if and only if $\sigma$ is separated and $S_1(X,\sigma)=\emptyset$.
\end{lem}

\begin{proof}
By definition of properness and separatedness of algebraic spaces, it is equivalent to work on a local chart $(U,u)$ of $X$. Since $X$ is a projective algebraic space, we also assume that $U$ is a projective space. Denote by $\sigma_u$ the induced group action on $U$. By \cite[Lemma 0.5]{MuFoKir}, we know that $\sigma_u$ is proper if and only if $\sigma_u$ is separated and $S_1(U,\sigma_u)$ is empty. This implies the lemma.
\end{proof}

\subsection{Geometric Invariant Theory on Algebraic Spaces}
Let $X$ be an algebraic space. A \emph{coherent} (resp. \emph{quasi-coherent}) \emph{sheaf} $F$ on $X$ is defined on local charts of $X$. More precisely, a coherent (resp. quasi-coherent) sheaf $F$ on $X$ is defined in the following way. On each local chart $(U,u)$ of $X$, let $F_u$ be a coherent (resp. quasi-coherent) sheaf on $U$. Let $\alpha_{uv}: F_u \rightarrow f_{uv}^* F_v$ be an isomorphism of coherent (resp. quasi-coherent) sheaves. The coherent (resp. quasi-coherent) sheaf $F$ is defined by the data $(F_u,\alpha_{uv})$. Sometimes we use the notation $F=(F_u,\alpha_{uv})$ to work on the coherent (resp. quasi-coherent) sheaf $F$ locally. Indeed, fixing a local chart $(U,u)$ of $X$, consider $s,t: U \times_X U \rightarrow U$. A coherent sheaf $F$ on $X$ is equivalent to a pair $(F_u,\gamma)$, where $F_u$ is a coherent sheaf on $U$ and $\gamma: s^*F_u \rightarrow t^* F_u$ is an isomorphism. Given a local chart $(U,u)$ and a coherent sheaf $F$ over $X$, if there is no ambiguity, we prefer to use the following notations
\begin{align*}
F(U)=F_U:=F_u
\end{align*}
for the coherent sheaf over $U$. A sheaf $F$ is \emph{locally free} if for each local chart $(U,u)$, the sheaf $F_u$ is locally free. A sheaf $F$ is \emph{invertible} if for each local chart $(U,u)$, the sheaf $F_u$ is invertible.

Let $r \in H^0(X,F)$ be a section of $F$. Denote by $r_u$ the restriction of $r$ to a local chart $(U,u)$ of $X$. Denote by $U_r$ the subscheme of $U$ such that $x \in U_r$ if $r(x) \neq 0$. Since $s^* r_u = t^* r_u$, we have $s^* U_r=t^* U_r \subseteq U \times_S U$. Note that the natural projections $s,t: s^* U_r \rightarrow U_r$ are \'etale. Therefore, we can define the algebraic space $X_r:= U_r/s^* U_r$ with respect to a given section $r \in H^0(X,F)$. It is easy to check that this definition is independent of the choice of local charts.

Let $\sigma: G \times_S X \rightarrow X$ be a $G$-action on $X$. A coherent sheaf $F$ on $X$ is \emph{$G$-linearized} if there is a morphism $\tau:\sigma^* F \rightarrow p_X^{*} F$ such that $\tau$ is an isomorphism and the following diagram commutes
\begin{center}
\begin{tikzcd}
\lbrack \sigma \circ (1_G \times \sigma) \rbrack ^* F \arrow[r, "(1_G \times \sigma)^* \tau"] \arrow[dd,equal]  & \lbrack p_X \circ (1_G \times \sigma)\rbrack ^* F \arrow[d]&  \\
& \lbrack \sigma \circ p_{23} \rbrack ^* F \arrow[r, "p_{23}^* \tau"] & \lbrack p_X \circ p_{23} \rbrack ^* F \arrow[d,equal] \\
\lbrack \sigma \circ (\mu \times 1_X) \rbrack ^* F \arrow[rr, "(\mu \times 1_X)^* \tau"] & & \lbrack p_X \circ (\mu \times 1_X) \rbrack ^* F
\end{tikzcd}
\end{center}
where $p_{23}: G \times_S G \times_S X \rightarrow G \times_S X$ is the projection omitting the first factor.

Let $r \in H^0(X,F)$ be a section of a coherent sheaf $F$ with a $G$-linearization $\tau$. We say that $r$ is \emph{$G$-invariant} if $\tau(\sigma^*(r))=p_X^*(r)$. If $r$ is a $G$-invariant section, the $G$-action $\sigma$ on $X$ induces a well-defined $G$-action on $X_r$.

\begin{defn}\label{216}
Let $L$ be an invertible sheaf on $X$ with a $G$-linearization $\tau$. Let $x$ be a geometric point on $X$.
\begin{enumerate}
\item $x$ is \emph{semistable} if there exists a $G$-invariant section $r \in H^0(X,L^n)$ for some $n$, such that $r(x) \neq 0$ and there exists a uniform categorical quotient of $X_r$.
\item $x$ is \emph{stable} if there exists a $G$-invariant section $r \in H^0(X,L^n)$ for some $n$, such that $r(x)\neq 0$, the action of $G$ on $X_r$ is separated and there exists a uniform categorical quotient of $X_r$.
\end{enumerate}
\end{defn}
Denote by $X^{ss}(L)$ (resp. $X^{s}(L)$) the set of semistable points (stable points) in $X$. The set $X^s(L)$ can be written as the union of disjoint sets $X^s_i(L)$, where $i$ means that the dimension of the stabilizer is $i$.

\begin{thm}\label{217}
Let $X$ be a quasi-compact algebraic space over $S$ with a group action $\sigma: G \times_S X \rightarrow X$. Given a $G$-linearized invertible sheaf $L$ on $X$, there exists a uniform categorical quotient $(Y^{ss}(L),\phi)$ of $X^{ss}(L)$. Furthermore, there is an open subset $Y^{s}(L) \subseteq Y^{ss}(L)$ such that $X^{s}(L)=\phi^{-1}(Y^{s}(L))$ and $Y^s(L)$ is a uniform good quotient of $X^s(L)$.
\end{thm}

\begin{proof}
Since $X$ is a quasi-compact algebraic space, there exists finitely many invariant sections $s_1, \dots, s_n$ of $L$ such that $X^{ss}(L)= \cup_{i=1}^n X_{s_i}$. By the definition of semistable points, each $X_{s_i}$ has a uniform categorical quotient $(Y_{s_i},\phi_i)$ with respect to the $G$-action. We will glue $Y_{s_i}$ together to get an algebraic space, which is the desired categorical quotient $Y^{ss}(L)$.

Let $(U_{s_i},u_{s_i})$ and $(V_{s_i},v_{s_i})$ are local charts of $X_{s_i}$ and $Y_{s_i}$ respectively such that $U_{s_i} \cong V_{s_i} \times_{Y_{s_i}} X_{s_i}$.
\begin{center}
\begin{tikzcd}
    U_{s_i}  \arrow[r,"u_{s_i}"] \arrow[d] & X_{s_i} \arrow[d] \\
    V_{s_i} \arrow[r,"v_{s_i}"]   & Y_{s_i}
\end{tikzcd}
\end{center}
Denote by $U^1_{s_i}:=U_{s_i} \times_{X_{s_i}}U_{s_i}$, $V^1_{s_i}:=V_{s_i} \times_{Y_{s_i}} V_{s_i}$. We use the same notation $s_i$ for the section on $(U_{s_i},u_{s_i})$, $1\leq i\leq n$. Define
\begin{align*}
V_{s_{ij}}=V_{s_i}-\{y | s_j(y)=0\}.
\end{align*}
Clearly, we have
\begin{align*}
\phi^{-1}_{s_i}(V_{s_{ij}}) = U_{s_i} \cap U_{s_j} =  \phi^{-1}_{s_j}(V_{s_{ji}}).
\end{align*}
By assumption, $(V_{s_i},\phi_{s_i})$ is a uniform categorical quotient of $U_{s_i}$ with respect to the induced $G$-action. Therefore, both $V_{ij}$ and $V_{ji}$ are categorical quotients of $U_{s_i} \cap U_{s_j}$. By the universal property of categorical quotients, there is a unique isomorphism $\psi_{ij}: V_{s_{ij}} \xrightarrow{\cong} V_{s_{ji}}$ making the diagram
\begin{center}
\begin{tikzcd}
    & U_{s_i} \cap U_{s_j}  \arrow[rd,"\phi_j"] &  \\
    V_{s_{ij}} \arrow[ur,"\phi_j"] \arrow[rr,"\psi_{ij}"]   & & V_{s_{ji}}
\end{tikzcd}
\end{center}
commutative. Under the isomorphisms $\psi_{ij}$, where $1 \leq i,j \leq n$, the schemes $V_{s_i}$, $1 \leq i \leq n$ can be glued together, and denote by $V$ the resulting scheme. At the same time, the isomorphism $\psi_{ij}$ also induces an isomorphism $V^1_{s_{ij}} \xrightarrow{\cong} V^{1}_{s_{ji}}$. With the same argument, we get a scheme $V^1$ by gluing $\{V^1_{s_{i}}\}_{1 \leq i \leq n}$ via $\{V^1_{s_{ij}}, 1 \leq i,j \leq n\}$. Denote by $Y^{ss}(L)$ the algebraic space $V/V^1$. By the construction above, $Y^{ss}(L)$ is a uniform categorical quotient of $X^{ss}(L)$.

For the stable locus $Y^s(L)$, the proof is the same. In this case, the action of $G$ on $X_{s_i}$ is separated. Therefore, by Proposition \ref{214}, a uniform good quotient exists.
\end{proof}

\begin{rem}\label{218}
Recall that the definition of semistability and stability in the case of schemes requires that $X_r$ is affine \cite[Definition 1.7]{MuFoKir}. One of the reasons is that in the affine case, a uniform categorical quotient exists \cite[Theorem 1.1]{MuFoKir}. In the case of algebraic spaces, if we assume that $U_r$ is affine, then $X_r$ is also an affine scheme \cite[Corollary 6.2.14]{Ol}. Then, it reduces to the classical case.
\end{rem}

The proof of Theorem \ref{217} also implies the following corollary.

\begin{cor}\label{219}
With the same assumption as in Theorem \ref{217}, there exists an invertible bundle $M$ on $Y^{ss}(L)$ such that $\phi^* M =L$.
\end{cor}

\subsection{Hilbert-Mumford Criterion}
J. Heinloth generalizes the Hilbert-Mumford criterion to algebraic stacks \cite{Hein}. In this subsection, we consider the special case on algebraic spaces. The setup in this subsection is different from previous ones, and in the rest of this paper, we will follow the setup in this subsection. Let $S$ be an algebraic space over an algebraically field $k$, and denote by $({\rm Sch/S})$ the category of $S$-schemes with respect to the big \'etale topology (or fppf topology).

Let $G$ be a group scheme (over $S$) with a proper action $\sigma$ on a proper algebraic space $X$. Let $\mathbb{G}_m$ be the multiplicative group scheme. An \emph{one-parameter subgroup} of $G$ is a homomorphism $\lambda: \mathbb{G}_m \rightarrow G$. Let $x: \Spec(k) \rightarrow X$ be a point in $X$. We consider the morphism
\begin{align*}
\lambda_x: \mathbb{G}_m \times_S \Spec(k) \xrightarrow{\lambda \times 1_x} G \times_S \Spec(k) \xrightarrow{\sigma_x} X.
\end{align*}
At the point $x$, we can identify $\mathbb{G}_m$ with $\Spec (k[t,t^{-1}])$. By valuative criterion for the properness of the action $\sigma$, we can extend the action to $\Spec (k[t])$. Denote by $\lambda_x(0)$ the specialization of $\lambda_x(t)$ as $t$ goes to zero. Now let $L$ be a $G$-linearization line bundle over $X$. With respect to the homomorphism $\lambda: \mathbb{G}_m \rightarrow G$, we consider the induced $\mathbb{G}_m$-linearization of $L$ restricted to the fixed point $\lambda_x(0)$. This linearization is given by the character of $\mathbb{G}_m$, $\chi(t)=t^r$, for $t \in \mathbb{G}_m$.
\begin{defn}\label{220}
Let $G$ be an algebraic space acting properly on an algebraic space $X$. Let $L$ be a $G$-linearization line bundle on $X$. We fix a closed point $x \in X$, and a one-parameter subgroup $\lambda$ of $G$. We define
\begin{align*}
\mu^L(x,\lambda)=-r.
\end{align*}
\end{defn}

\begin{prop}\label{221}
Let $G$ be a reductive group acting properly on an algebraic space $X$ proper over $S$. Let $L$ be an ample $G$-linearized line bundle on $X$. Then, for any one-parameter subgroup $\lambda$,
\begin{itemize}
\item $x \in X^{ss}(L)$ if and only $\mu^L(x,\lambda) \geq 0$;
\item $x \in X^{s}(L)$ if and only $\mu^L(x,\lambda) > 0$.
\end{itemize}
\end{prop}

\begin{proof}
The existence of $X^{ss}(L)$ is proved in Theorem \ref{217}, and the proposition can be proved in the same way as \cite[Theorem 2.1]{MuFoKir}. This proposition is also a special case of J. Heinloth's construction (see \cite[\S 1]{Hein} for details).
\end{proof}

Let $\widetilde{{\rm Grass}_S}(r,n): ({\rm Sch}/S)^{\rm op} \rightarrow ({\rm Set})$ be the Grassmannian functor such that for each $S$-scheme $T$, $\widetilde{{\rm Grass}_S}(r,n)(T)$ is the set of $r$-dimensional vector bundles $V_T$ over $T$ together with a linear inclusion $V_T \rightarrow T \otimes k^n$. Two inclusions are equivalent if there is a commutative diagram
\begin{center}
\begin{tikzcd}
V_T \arrow[r, hook] \arrow[d,"\cong"]  & T \times k^n \arrow[d,equal] \\
V'_T \arrow[r, hook]& T \times k^n
\end{tikzcd}
\end{center}
The Grassmannian functor is represented by an algebraic space over $S$, and we denote by ${\rm Grass}_S(r,n)$. There is a canonical invertible sheaf $\mathcal{O}(1)$ such that its restriction to a point $P \in {\rm Grass}_S(r,n)$ is isomorphic to $\wedge^r P$.

There is a natural $\text{SL}(n)$-action on $\text{Grass}_S(r,n)$, which induces a natural action on $\mathcal{O}(1)$.

\begin{cor}\label{222}
A point $P \in {\rm Grass}_S(r,n)$ is semistable with respect to the ${\rm SL}(V)$-action and the line bundle $\mathcal{O}(1)$, if and only if for any linear subspace $L \subseteq S^n$,
\begin{align*}
\frac{\dim(P \cap L)+1 }{r+1} < \frac{\dim(P)+1}{n+1}.
\end{align*}
\end{cor}

\begin{proof}
Let $U \rightarrow S$ be a surjective \'etale morphism. Then, the induced morphism ${\rm Grass}_U(r,n) \rightarrow {\rm Grass}_S(r,n)$ is also surjective \'etale, which is a local chart of the algebraic space ${\rm Grass}_S(r,n)$. Therefore, by Theorem \ref{206}, it is equivalent to prove the statement on ${\rm Grass}_U(r,n)$, which is a scheme, and this is proved as a special case in \cite[Proposition 4.3]{MuFoKir}.
\end{proof}

\section{Moduli Spaces of Coherent Sheaves on Projective Deligne-Mumford Stacks}
\subsection{Preliminaries}
Let $S$ be an algebraic space, which is locally of finite type over an algebraically closed field $k$. A Deligne-Mumford stack $\mathcal{X}$ over $S$ is a morphism $\mathcal{X} \rightarrow S$, which is also considered as a family of Deligne-Mumford stacks over $S$. Suppose that $\mathcal{X}$ has a \emph{coarse moduli space} $X$ (as an algebraic space). Denote by $\pi: \mathcal{X} \rightarrow X$ the natural morphism of algebraic spaces. Note that there is a natural morphism $\rho:X \rightarrow S$. The stack $\mathcal{X}$ is \emph{tame} if the functor
\begin{align*}
\pi_*: {\rm QCoh}(\mathcal{X}) \rightarrow {\rm QCoh}(X)
\end{align*}
is exact, where ${\rm QCoh}(-)$ is the category of quasi-coherent sheaves.

Now we fix a locally free sheaf $\mathcal{E}$ on $\mathcal{X}$, and define two functors
\begin{align*}
F_{\mathcal{E}} : {\rm QCoh}(\mathcal{X}) \rightarrow {\rm QCoh}(X), \quad G_{\mathcal{E}}:  {\rm QCoh}(X) \rightarrow {\rm QCoh}(\mathcal{X})
\end{align*}
as follows
\begin{align*}
F_{\mathcal{E}}(\mathcal{F})= \pi_* \mathcal{H}om_{\mathcal{O}_{\mathcal{X}}}(\mathcal{E},\mathcal{F}), \quad G_{\mathcal{E}}(F)=\pi^*F \otimes \mathcal{E},
\end{align*}
where $\mathcal{F} \in {\rm QCoh}(\mathcal{X})$ and $F \in {\rm QCoh}(X)$. The functor $F_{\mathcal{E}}$ is exact since $\pi_*$ is exact and $\mathcal{E}$ is a locally free sheaf. However, the functor $G_{\mathcal{E}}$ may not be exact. The composition of the two functors
\begin{align*}
G_{\mathcal{E}} \circ F_{\mathcal{E}}: {\rm QCoh}(\mathcal{X}) \rightarrow {\rm QCoh}(\mathcal{X}),
\end{align*}
is
\begin{align*}
G_{\mathcal{E}} \circ F_{\mathcal{E}}(\mathcal{F})=\pi^* \pi_* \mathcal{H}om_{\mathcal{O}_{\mathcal{X}}}(\mathcal{E},\mathcal{F})\otimes \mathcal{E}.
\end{align*}
Denote by
\begin{align*}
\theta_{\mathcal{E}}(\mathcal{F}): G_{\mathcal{E}} \circ F_{\mathcal{E}}(\mathcal{F}) \rightarrow \mathcal{F}.
\end{align*}
the adjunction morphism left adjoint to the identity morphism $\pi_*(\mathcal{F} \otimes \mathcal{E}^{\vee}) \xrightarrow{\rm id} \pi_*(\mathcal{F} \otimes \mathcal{E}^{\vee})$.

\begin{defn}\label{301}
A locally free sheaf $\mathcal{E}$ is a \emph{generator} for $\mathcal{F} \in {\rm QCoh}(\mathcal{X})$, if the morphism
\begin{align*}
\theta_{\mathcal{E}}(\mathcal{F}): \pi^* \pi_* \mathcal{H}om_{\mathcal{O}_{\mathcal{X}}}(\mathcal{E},\mathcal{F})\otimes \mathcal{E} \rightarrow \mathcal{F}
\end{align*}
is surjective. A locally free sheaf $\mathcal{E}$ is a \emph{generating sheaf} of $\mathcal{X}$, if it is a generator for every quasi-coherent sheaf on $\mathcal{X}$.
\end{defn}

A Deligne-Mumford stack $\mathcal{X}$ over $S$ is a \emph{global quotient} if it is isomorphic to a stack $[U/G]$ where $U$ is an algebraic space of finite type over $S$ and $G \rightarrow S$ is a flat group scheme, which is a subgroup scheme of $\text{GL}_{N,S}$ for some integer $N$. Olsson and Starr proved that if $\mathcal{X}$ is a global quotient, then there exists a generating sheaf of $\mathcal{X}$.

\begin{thm}[Theorem 5.7 in \cite{OlSt}]\label{302}
Let $\mathcal{X}$ be a tame Deligne-Mumford stack, which is a separated global quotient over $S$, then there exists a generating sheaf $\mathcal{E}$ over $\mathcal{X}$.
\end{thm}

Now we consider a special Deligne-Mumford stack.
\begin{defn}\label{303}
A \emph{projective Deligne-Mumford stack} $\mathcal{X}$ over $S$ is a tame Deligne-Mumford stack $p: \mathcal{X} \rightarrow S$, which is a separated locally finitely-presented global quotient, such that the coarse moduli space $X$ is projective over $S$.
\end{defn}
We fix a locally free sheaf $\mathcal{H}$, not necessarily to be a generating sheaf. Furthermore, we fix a polarization $\mathcal{O}_X(1)$ over $X$, where a polarization is a relatively ample invertible bundle on $X$ with respect to the projective morphism $X \rightarrow S$.

Let $\mathcal{F}$ be a coherent sheaf on $\mathcal{X}$. The \emph{modified Hilbert polynomial} $P_{\mathcal{H}}(\mathcal{F},m)$ is defined as
\begin{align*}
P_{\mathcal{H}}(\mathcal{F},m)=\chi(\mathcal{X},\mathcal{F} \otimes \mathcal{H}^{\vee} \otimes \pi^* \mathcal{O}_X(1)^m), \quad m \gg 0.
\end{align*}
Since the functor $\pi_*: {\rm QCoh}(\mathcal{X}) \rightarrow {\rm QCoh}(X)$ is exact, the modified Hilbert polynomial can be written as the classical Hilbert polynomial for the coherent sheaf $F_{\mathcal{H}}(\mathcal{F})(m)$ over $X$,
\begin{align*}
P_{\mathcal{H}}(\mathcal{F},m)=\chi(X,F_{\mathcal{H}}(\mathcal{F})(m)),\quad  m \gg 0,
\end{align*}
where $F_{\mathcal{H}}(\mathcal{F})(m)=F_{\mathcal{H}}(\mathcal{F}) \otimes \mathcal{O}_X(m)$. If $\mathcal{F}$ is pure of dimension $d$, the function $P_{\mathcal{H}}(\mathcal{F},m)$ is a polynomial (with respect to the variable $m$) and we can write it in the following way
\begin{align*}
P_{\mathcal{H}}(\mathcal{F},m)=\sum_{i=0}^d \alpha_{\mathcal{H},i}(\mathcal{F})\frac{m^i}{i!}.
\end{align*}
We use the notation $P_{\mathcal{H}}(\mathcal{F})$ for the modified Hilbert polynomial of $\mathcal{F}$. The \emph{reduced modified Hilbert polynomial} $p_{\mathcal{H}}(\mathcal{F})$ is a monic polynomial with rational coefficients defined as
\begin{align*}
p_{\mathcal{H}}(\mathcal{F})=\frac{P_{\mathcal{H}}(\mathcal{F})}{\alpha_{\mathcal{H},d}(\mathcal{F})}.
\end{align*}
A pure sheaf $\mathcal{F}$ is \emph{$p_{\mathcal{H}}$-semistable} (resp. \emph{$p_{\mathcal{H}}$-stable}), if for every proper subsheaf $\mathcal{F}' \subseteq \mathcal{F}$ we have
\begin{align*}
p_{\mathcal{H}}(\mathcal{F}') \leq p_{\mathcal{H}}(\mathcal{F})\quad (\text{resp. } p_{\mathcal{E}}(\mathcal{F}') < p_{\mathcal{H}}(\mathcal{F})).
\end{align*}
If $\mathcal{H}$ is a generating sheaf, this definition is exactly what Nironi defined in \cite[\S 3.2]{Nir}.

With respect to the above definition, Jordan-H\"older filtrations exist for $\mathcal{H}$-semistable coherent sheaves, which induce the $S$-equivalence condition for $\mathcal{H}$-semistable coherent sheaves. The construction of Jordan-H\"older filtrations and proof of the existence is exactly the same as the case that $\mathcal{H}$ is a generating sheaf (\cite[\S 3.4]{Nir} and \cite[\S 3.4]{Sun2020}).

\subsection{Quot-functors}
Let $S$ be a locally of finite type algebraic space over an algebraically closed field $k$. Denote by $(\text{Sch}/S)$ the site of $S$-schemes with respect to the big \'etale topology. Let $\mathcal{X}$ be a separated and locally finitely-presented Deligne-Mumford stack over $S$. We take a coherent sheaf $\mathcal{G}$ on $\mathcal{X}$. We define the moduli problem
\begin{align*}
\widetilde{{\rm Quot}}_S(\mathcal{G},\mathcal{X}): (\text{Sch}/S)^{{\rm op}} \rightarrow \text{Set}
\end{align*}
as follows. For each $S$-scheme $T$, define $\mathcal{X}_T$ as $\mathcal{X} \times_{S} T$ and $\mathcal{G}_T$ the pullback of $\mathcal{G}$ to $\mathcal{X}_T$. Define $\widetilde{{\rm Quot}}_S(\mathcal{G},\mathcal{X})(T)$ to be the set of $\mathcal{O}_{\mathcal{X}_T}$-module quotients $[\mathcal{G}_T \rightarrow \mathcal{F}]$ such that
\begin{enumerate}
\item $\mathcal{F}$ is a locally finitely-presented quasi-coherent sheaf;
\item $\mathcal{F}$ is flat over $T$;
\item the support of $\mathcal{F}$ is proper over $T$.
\end{enumerate}
The moduli problem $\widetilde{{\rm Quot}}_S(\mathcal{G},\mathcal{X})$ is called the \emph{quot-functor}. The quot-functor $\widetilde{{\rm Quot}}_S(\mathcal{G},\mathcal{X})$ has a natural stack structure. In other words, $\widetilde{{\rm Quot}}_S(\mathcal{G},\mathcal{X})$ is a sheaf over $(\text{Sch}/S)$.

Artin proved that the quot-functor $\widetilde{{\rm Quot}}_S(\mathcal{G},\mathcal{X})$ is represented by a separated and locally finitely-presented algebraic space over $S$ when $\mathcal{X}$ is an algebraic space \cite{Art}. Olsson and Starr generalized this result to Deligne-Mumford stacks.
\begin{thm}[Theorem 1.1 in \cite{OlSt}]\label{304}
With respect to the above notation, the quot-functor $\widetilde{{\rm Quot}}_S(\mathcal{G},\mathcal{X})$ is represented by an algebraic space which is separated and locally finitely presented over $S$.
\end{thm}
Denote by ${\rm Quot}_S(\mathcal{G},\mathcal{X})$ the algebraic space representing $\widetilde{{\rm Quot}}_S(\mathcal{G},\mathcal{X})$. We would like to use the notations $\widetilde{{\rm Q}}(\mathcal{G},\mathcal{X})$ for the Quot-functor and ${\rm Q}(\mathcal{G},\mathcal{X})$ for the corresponding algebraic space, and if there is no ambiguity, we also omit the subscript $S$.

Now let $\mathcal{X}$ be a projective Deligne-Mumford stack over $S$. We fix an integer polynomial $P$. Define $\widetilde{{\rm Q}}(\mathcal{G},\mathcal{X},P)$ to be the sub-functor of $\widetilde{{\rm Q}}(\mathcal{G},\mathcal{X})$ such that for each $S$-scheme $T$, $\widetilde{{\rm Q}}(\mathcal{G},\mathcal{X},P)(T)$ consists of all quotients $[\mathcal{G}_T \rightarrow \mathcal{F}] \in \widetilde{{\rm Q}}(\mathcal{G},\mathcal{X})(T)$, of which the modified Hilbert polynomial is $P$. Denote by the algebraic space ${\rm Q}(\mathcal{G},\mathcal{X},P)$ representing $\widetilde{{\rm Q}}(\mathcal{G},\mathcal{X},P)$. Furthermore, ${\rm Q}(\mathcal{G},\mathcal{X},P)$ is a projective scheme over $S$ when $S$ is an affine scheme.
\begin{thm}[Theorem 1.5 in \cite{OlSt}]\label{305}
Suppose that $S$ is an affine scheme and $\mathcal{X}$ is a projective Deligne-Mumford stack over $S$. The algebraic space ${\rm Q}(\mathcal{G},\mathcal{X},P)$ is a projective scheme over $S$.
\end{thm}
This theorem implies the following general result when $S$ is an algebraic space.
\begin{prop}\label{306}
Let $S$ be a locally of finite type algebraic space over an algebraically closed field $k$, and let $\mathcal{X}$ be a projective Deligne-Mumford stack over $S$. We fix an integer polynomial $P$ and a coherent sheaf $\mathcal{G}$ over $\mathcal{X}$. Then the algebraic space ${\rm Q}(\mathcal{G},\mathcal{X},P)$ is projective over $S$.
\end{prop}

\begin{proof}
To prove that the algebraic space ${\rm Q}_S(\mathcal{G},\mathcal{X},P)$ is projective over $S$, it is enough to work on an affine scheme $U$ such that $u: U \rightarrow S$ is an \'etale morphism, and prove that ${\rm Q}_U(\mathcal{G}_U,\mathcal{X}_U,P)$ is projective over $U$. Since $X \rightarrow S$ is a projective morphism of algebraic spaces, $X_U$ is also projective over $U$. By Theorem \ref{305}, ${\rm Q}_U(\mathcal{G}_U,\mathcal{X}_U,P)$ is a $U$-projective scheme. This finishes the proof of this proposition.
\end{proof}

Next, we will show that the quot-functor $\widetilde{{\rm Q}}(\mathcal{G},\mathcal{X},P)$ has a universal family. In the classical discussion for moduli problems, if a moduli problem
\begin{align*}
\widetilde{\mathcal{M}}: ({\rm Sch}/S) \rightarrow {\rm Set}
\end{align*}
is representable by $M$ (as schemes), then a universal family exists and corresponds to the identity in $\widetilde{\mathcal{M}}(M)={\rm Hom}(M,M)$. However, in our case, the moduli problem $\widetilde{{\rm Q}}(\mathcal{G},\mathcal{X},P)$ is represented by an algebraic space ${\rm Q}(\mathcal{G},\mathcal{X},P)$, which is not a scheme generally. Therefore, $\widetilde{{\rm Q}}(\mathcal{G},\mathcal{X},P)({\rm Q}(\mathcal{G},\mathcal{X},P))$ does not make sense by definition. Therefore, we have to construct a ``universal family" corresponding to the identity morphism in ${\rm Hom}({\rm Q}(\mathcal{G},\mathcal{X},P),{\rm Q}(\mathcal{G},\mathcal{X},P))$.

\begin{thm}\label{307}
There exists a universal family over $X \times_S {\rm Q}(\mathcal{G},\mathcal{X},P)$.
\end{thm}

\begin{proof}
We use the notation $Q:={\rm Q}(\mathcal{G},\mathcal{X},P)$ for the algebraic space and $\widetilde{Q}:=\widetilde{{\rm Q}}(\mathcal{G},\mathcal{X},P)$ for the functor. Let $q: Q_0 \rightarrow Q$ be an \'etale surjective morphism, where $Q_0$ is a scheme. Let $Q_1:=Q_0 \times_{Q} Q_0$, and denote by $s,t: Q_1 \rightarrow Q_0$ the source and target map respectively. Clearly, $Q \cong Q_0 / Q_1$, which is induced by the morphism $q: Q_0 \rightarrow Q$. Denote by $\mathcal{Q}_0$ the element in $\widetilde{Q}(Q_0)$ corresponding to the morphism $q: Q_0 \rightarrow {\rm Q}(\mathcal{G},\mathcal{X})$, and $\mathcal{Q}_0$ is a coherent sheaf on $Q_0 \times_S \mathcal{X}$ by definition. Since $s \circ q=t \circ q$, there is a canonical isomorphism $\sigma_{can}: s^* \mathcal{Q}_0 \cong t^* \mathcal{Q}_0$. This gives us a coherent sheaf $\mathcal{Q}$ on $Q(\mathcal{G},\mathcal{X}) \times_S \mathcal{X}$. Note that the construction of the coherent sheaf $\mathcal{Q}$ is independent of the choice of the surjective \'etale morphism $q: Q_0 \rightarrow Q$. We will prove that $\mathcal{Q}$ is the universal family. For simplicity, we say that $\mathcal{Q}$ is a family of coherent sheaves on $Q$, and the same for the other families.

To prove that $\mathcal{Q}$ is the universal family, it is equivalent to prove that for any family of coherent sheaves $\mathcal{Q}_T$ on any $S$-scheme $T$, there is a unique morphism $\mathcal{Q}_T \rightarrow \mathcal{Q}$. Equivalently, $\mathcal{Q}_T$ is the pullback of the family $\mathcal{Q}$.
\begin{center}
\begin{tikzcd}
    \mathcal{Q}_T \arrow[r] \arrow[d] & \mathcal{Q} \arrow[d] \\
    T \arrow[r]   & Q
\end{tikzcd}
\end{center}
Let $T$ be an $S$-scheme. Let $p: T \rightarrow Q$ be a morphism, which corresponds to a family of coherent sheaves $\mathcal{Q}_T$ on $T$ by the representability of $Q$. We consider the following diagram
\begin{center}
\begin{tikzcd}
    T_{Q_1} \arrow[r,shift right=0.5ex] \arrow[r,shift left=0.5ex] \arrow[d] & T_{Q_0} \arrow[d] \arrow[r] & T \arrow[d,"p"]\\
    Q_1 \arrow[r,shift right=0.5ex] \arrow[r,shift left=0.5ex]  & Q_0 \arrow[r,"q"] & Q
\end{tikzcd}
\end{center}
where $T_{Q_0}= T \times_Q Q_0$, $T_{Q_1}= T \times_Q Q_1$, and each square is cartesian. Since $\widetilde{Q}$ is a sheaf (see \cite[\S 2]{OlSt}), we have an exact sequence
\begin{align*}
0 \rightarrow \widetilde{Q}(T) \rightarrow \widetilde{Q}(T_{Q_0}) \rightrightarrows \widetilde{Q}(T_{Q_1}).
\end{align*}
Note that the element $\mathcal{Q}_0 \in \widetilde{Q}(Q_0)$ will be mapped to a trivial element in $\widetilde{Q}(T_{Q_1})$ via the morphism
\begin{align*}
\widetilde{Q}(Q_0) \rightrightarrows  \widetilde{Q}(Q_1) \rightarrow \widetilde{Q}(T_{Q_1}).
\end{align*}
By the universal property of the injective morphism,
\begin{center}
\begin{tikzcd}
    0 \arrow[r] & \widetilde{Q}(T) \arrow[r] & \widetilde{Q}(T_{Q_0}) \arrow[r,shift right=0.5ex] \arrow[r,shift left=0.5ex] & \widetilde{Q}(T_{Q_1})\\
    & & \widetilde{Q}(Q_0) \arrow[u] \arrow[ul, dotted, "\exists !" description] &
\end{tikzcd}
\end{center}
$\mathcal{Q}_0$ will be mapped to a unique element in $\widetilde{Q}(T)$. This element is exactly $\mathcal{Q}_T$ (corresponding to the map $p: T \rightarrow Q$), which is determined by the fiber product
\begin{center}
\begin{tikzcd}
    T_{Q_0} \arrow[d] \arrow[r] & T \arrow[d,"p"]\\
    Q_0 \arrow[r,"q"] & Q
\end{tikzcd}
\end{center}
Therefore, we have a unique morphism $\mathcal{Q}_T \rightarrow \mathcal{Q}_0$. Since $\mathcal{Q}_0 \in \widetilde{Q}(Q_0)$ maps to the trivial element in $\widetilde{Q}(Q_1)$, we induce a morphism $\mathcal{Q}_T \rightarrow \mathcal{Q}$. This finishes the proof that $\mathcal{Q}$ is a universal family.
\end{proof}
Denote by $\mathcal{Q}$ the universal family over ${\rm Q}(\mathcal{G},\mathcal{X},P) \times_S \mathcal{X}$.

With the same setup as in Proposition \ref{306}, denote by $X$ the coarse moduli space of $\mathcal{X}$, which is an algebraic space. As an application of the proposition, the algebraic space ${\rm Q}(G,X,P)$ is projective over $S$, where $G$ is a coherent sheaf on $X$. Let $\mathcal{E}$ be a generating sheaf of $\mathcal{X}$. The functor $F_{\mathcal{E}}$ induces a closed immersion ${\rm Q}(\mathcal{G},\mathcal{X},P) \hookrightarrow {\rm Q}(F_{\mathcal{E}}(\mathcal{G}),X,P)$. The proof of this result is the same as \cite[\S 6]{OlSt}. Now let $\mathcal{H}$ be a locally free sheaf on a tame Deligne-Mumford stack $\mathcal{X}$. In this case, the morphism
\begin{align*}
\theta_{\mathcal{H}}(\mathcal{F}): \pi^* \pi_* \mathcal{H}om(\mathcal{H},\mathcal{F}) \otimes \mathcal{H} \rightarrow \mathcal{F}
\end{align*}
is not necessarily to be surjective.

We consider the following moduli problem
\begin{align*}
\widetilde{{\rm Q}}^{\mathcal{H}}(\mathcal{G},\mathcal{X},P): (\text{Sch}/S)^{{\rm op}} \rightarrow \text{Set}.
\end{align*}
For each $S$-scheme $T$, $\widetilde{{\rm Q}}^{\mathcal{H}}(\mathcal{G},\mathcal{X},P)(T) \subseteq \widetilde{{\rm Q}}(\mathcal{G},\mathcal{X},P)(T)$ consists of all quotients $[q: \mathcal{G}_T \rightarrow \mathcal{F}]$ such that the morphism $\theta_{\mathcal{H}}({\rm ker}(q))$ is surjective. Clearly, if $\mathcal{X}$ is a projective Deligne-Mumford stack and $\mathcal{H}$ is a generating sheaf, then the moduli problem $\widetilde{{\rm Q}}^{\mathcal{H}}(\mathcal{G},\mathcal{X},P)$ is exactly $\widetilde{{\rm Q}}(\mathcal{G},\mathcal{X},P)$.

\begin{prop}\label{308}
The moduli problem $\widetilde{{\rm Q}}^{\mathcal{H}}(\mathcal{G},\mathcal{X},P)$ is represented by an algebraic space ${\rm Q}^{\mathcal{H}}(\mathcal{G},\mathcal{X},P)$, which is an open subset of ${\rm Q}(\mathcal{G},\mathcal{X},P)$.
\end{prop}

\begin{proof}
The proof of this proposition is inspired by \cite[Lemma 2.6]{Lieb}. By Proposition \ref{307}, the universal family $\mathcal{Q}$ on ${\rm Q}(\mathcal{G},\mathcal{X},P) \times_S \mathcal{X}$ corresponds to the identity morphism in ${\rm Hom}({\rm Q}(\mathcal{G},\mathcal{X},P),{\rm Q}(\mathcal{G},\mathcal{X},P))$. We have the following morphism of $\mathcal{Q}$
\begin{align*}
\theta_{\mathcal{H}}(\mathcal{Q}): \pi^* \pi_* \mathcal{H}om(\mathcal{H},\mathcal{Q}) \otimes \mathcal{H} \rightarrow \mathcal{Q}.
\end{align*}
Denote by $\mathcal{K}$ the kernel of the morphism $\theta_{\mathcal{H}}(\mathcal{Q})$. Proving the proposition is equivalent to prove that the locus of points $s \in {\rm Q}(\mathcal{G},\mathcal{X},P)$ such that $\mathcal{K}_{s}=0$ is open. Note that the locus of points $s \in {\rm Q}(\mathcal{G},\mathcal{X},P)$ such that $\mathcal{K}_{s}=0$ is exactly the complement of the support of $\mathcal{K}$. We also know that the support of $\mathcal{K}$ is proper over ${\rm Q}(\mathcal{G},\mathcal{X},P)$, which means that the support is closed. Therefore, the locus of points $s \in {\rm Q}(\mathcal{G},\mathcal{X},P)$ such that $\mathcal{K}_{s}=0$ is open. This finishes the proof of this proposition.
\end{proof}

Now we consider another important property of $\widetilde{{\rm Q}}^{\mathcal{H}}(\mathcal{G},\mathcal{X},P)$. The functor $F_{\mathcal{H}}: {\rm Qcoh}(\mathcal{X}) \rightarrow {\rm Qcoh}(X)$ induces a natural transformation
\begin{align*}
F_{\mathcal{H}}: \widetilde{{\rm Q}}(\mathcal{G},\mathcal{X},P) \rightarrow \widetilde{{\rm Q}}(F_{\mathcal{H}}(\mathcal{G}),\mathcal{X},P).
\end{align*}
This morphism is not a monomorphism in general, and it is a monomorphism when $\mathcal{H}$ is a generating sheaf \cite[Lemma 6.1]{OlSt}. Furthermore, we can prove that when restricting to $\widetilde{{\rm Q}}^{\mathcal{H}}(\mathcal{G},\mathcal{X},P)$, the natural transformation $F_{\mathcal{H}}$ is a monomorphism.

\begin{lem}\label{309}
The functor $F_{\mathcal{H}}$ induces a monomorphism
\begin{align*}
F_{\mathcal{H}}: \widetilde{{\rm Q}}^{\mathcal{H}}(\mathcal{G},\mathcal{X},P) \rightarrow \widetilde{{\rm Q}}(F_{\mathcal{H}}(\mathcal{G}),\mathcal{X},P).
\end{align*}
In other words, for each $S$-scheme $T$, the map
\begin{align*}
F_{\mathcal{H}}(T): \widetilde{{\rm Q}}^{\mathcal{H}}(\mathcal{G},\mathcal{X},P)(T) \rightarrow \widetilde{{\rm Q}}(F_{\mathcal{H}}(\mathcal{G}),\mathcal{X},P)(T)
\end{align*}
is injective.
\end{lem}

\begin{proof}
We first construct a map
\begin{align*}
\eta: \widetilde{{\rm Q}}(F_{\mathcal{H}}(\mathcal{G}),\mathcal{X},P) \rightarrow \widetilde{{\rm Q}}^{\mathcal{H}}(\mathcal{G},\mathcal{X},P).
\end{align*}
Then, we prove that the composition $F_{\mathcal{H}} \circ \nu$ is the identity. To simplify the proof, we only prove that the map $F_{\mathcal{H}}:=F_{\mathcal{H}}(S)$ is injective.

Let $\alpha: F_{\mathcal{H}}(\mathcal{G}) \rightarrow F$ be an element in $\widetilde{{\rm Q}}(F_{\mathcal{H}}(\mathcal{G}),\mathcal{X},P)$. Denote by $\beta: K \rightarrow F_{\mathcal{H}}(\mathcal{G})$ the kernel of $\alpha$. We define $\eta(\alpha)$ to be the cokernel of the following composition
\begin{align*}
\pi^*K \otimes \mathcal{H} \xrightarrow{\pi^* \beta \otimes 1} G_{\mathcal{H}}(F_{\mathcal{H}}( \mathcal{G})) \xrightarrow{\theta_{\mathcal{H}}(\mathcal{G})} \mathcal{G}.
\end{align*}

Now let $\gamma: \mathcal{G} \rightarrow \mathcal{F}$ be an element in $\widetilde{{\rm Q}}^{\mathcal{H}}(\mathcal{G},\mathcal{X},P)$. Denote by $\delta: \mathcal{K}\rightarrow \mathcal{G}$ the kernel of $\gamma$. By the exactness of the functor $F_{\mathcal{H}}$, we have a short exact sequence
\begin{align*}
0 \rightarrow F_{\mathcal{H}}(\mathcal{K}) \rightarrow F_{\mathcal{H}}(\mathcal{G}) \rightarrow F_{\mathcal{H}}(\mathcal{F}) \rightarrow 0.
\end{align*}
Let $\alpha=F_{\mathcal{H}}(\gamma)$ and $\beta=F_{\mathcal{H}}(\delta)$. Since $\pi^*$ is right exact, we have the following commutative diagram
\begin{center}
\begin{tikzcd}
    & G_{\mathcal{H}}(F_{\mathcal{H}}(\mathcal{K})) \arrow[r,"\pi^* \beta \otimes 1"]  \arrow[d,"\theta_{\mathcal{H}}(\mathcal{K})"] & G_{\mathcal{H}}(F_{\mathcal{H}}(\mathcal{G})) \arrow[d,"\theta_{\mathcal{H}}(\mathcal{G})"] \arrow[r,"\pi^*\alpha \otimes 1"] & G_{\mathcal{H}}(F_{\mathcal{H}}(\mathcal{F})) \arrow[d,"\theta_{\mathcal{H}}(\mathcal{F})"] \arrow[r] & 0 \\
    0 \arrow[r] & \mathcal{K} \arrow[r,"\delta"]  & \mathcal{G} \arrow[r,"\gamma"] & \mathcal{F} \arrow[r] & 0 .
\end{tikzcd}
\end{center}
Since $\theta_{\mathcal{H}}(\mathcal{G}) \circ \pi^*\beta \otimes 1 = \delta \circ \theta_{\mathcal{H}}(\mathcal{K})$ and $\theta_{\mathcal{H}}(\mathcal{K})$ is surjective, the image of $\theta_{\mathcal{H}}(\mathcal{G}) \circ \pi^*\beta \otimes 1$ is $\delta(\mathcal{K})$. Therefore, $\eta(F_{\mathcal{H}}(\gamma))=\gamma$. This finishes the proof of this lemma.
\end{proof}
The proof of this lemma is similar to that of \cite[Lemma 6.1]{OlSt}. In \cite[Lemma 6.1]{OlSt}, the morphisms $\theta_{\mathcal{H}}(\mathcal{K})$, $\theta_{\mathcal{H}}(\mathcal{G})$ and $\theta_{\mathcal{H}}(\mathcal{F})$ are all surjective, while we only have the first morphism $\theta_{\mathcal{H}}(\mathcal{K})$ to be surjective in Lemma \ref{309}. This surjection comes from the definition of $\widetilde{\rm Q}^{\mathcal{H}}(\mathcal{G},\mathcal{X},P)$.

\begin{cor}\label{310}
The monomorphism $F_{\mathcal{H}}: {\rm Q}^{\mathcal{H}}(\mathcal{G},\mathcal{X},P) \rightarrow {\rm Q}(F_{\mathcal{H}}(\mathcal{G}),X,P)$ is a finitely-presented closed immersion.
\end{cor}

\begin{proof}
Based on Lemma \ref{309}, this proof is the same as \cite[Proposition 6.2]{OlSt}.
\end{proof}

\subsection{Boundedness}
Let $X$ be an algebraic space over $S$, where $S$ is an algebraic space over an algebraically closed field $k$. A set-theoretic family of coherent sheaves $\mathfrak{F}$ on $X \rightarrow S$ is a set of coherent sheaves defined on the fibers of $X \rightarrow S$. More precisely, the coherent sheaves are defined on $X_s$, where $s = \Spec (k)$ is a point of $S$. Denote by $F_s$ an element in $\mathfrak{F}$, where the subscript $s$ means that the coherent sheaf $F_s$ is defined on the fiber $X_s$.
\begin{defn}\label{311}
A set-theoretic family $\mathfrak{F}$ of coherent sheaves on the algebraic space $X \rightarrow S$ is bounded if there is an $S$-scheme $T$, and a coherent sheaf $F_T$ on $X_T:=X \times_S T$ such that the family $\mathfrak{F}$ is contained in the set of fibers of $F_T$.
\end{defn}

Let $X \rightarrow S$ be a projective morphism of algebraic spaces, and let $s=\Spec(k)$ be a point of $S$. Note that the morphism $X \rightarrow S$ is representable by schemes. Therefore, the fiber $X_s$ is a projective scheme for any point $s \in S$. Based on this property, a set-theoretic family of coherent sheaves on $X \rightarrow S$ is exactly a set of coherent sheaves on projective schemes, which are parameterized by algebraic spaces.

\begin{lem}\label{312}
Let $X$ be an algebraic space projective over $S$. A set-theoretic family $\mathfrak{F}$ of coherent sheaves on $X \rightarrow S$ can be considered as a set-theoretic family on $X_U \rightarrow U$, where $U$ is a scheme with an \'etale morphism $U \rightarrow S$ and $X_U:=X \times_S U$ is the pullback.
\end{lem}

\begin{proof}
Let $U \rightarrow S$ be a surjective \'etale morphism, where $U$ is a scheme over $k$. Let $s=\Spec(k) \in S$ be a point. We have the following diagram and each square is cartesian.
\begin{center}
\begin{tikzcd}
    & X_U \times_U (U \times_S s) \arrow[ld] \arrow[rr] \arrow[dd] & & X_s \arrow[dd] \arrow[ld]\\
    X_U \arrow[rr] \arrow[dd]  & & X \arrow[dd] &\\
    & U \times_S s  \arrow[rr] \arrow[ld] & & s \arrow[ld] \\
    U \arrow[rr] & & S &
\end{tikzcd}
\end{center}
Since $U \rightarrow S$ is \'etale and surjective, $U \times_S s \rightarrow s$ is also \'etale and surjective. We know that $s=\Spec(k)$ and $k$ is an algebraically closed field. Thus, $U \times_S s$ is a product of finitely many points, i.e. $\Spec(k)$, in $U$. Note that  $X_U \times_U (U \times_S s) \cong (X_s \times_s (U \times_S s))$, which means that $X_U \times_U (U \times_S s)$ is a product of finitely many fibers $X_u$, where each fiber $X_u$ is isomorphic to $X_s$. Therefore, any element $F_s \in \mathfrak{F}$ can be considered as an  element over $(X_U)_u$, where $u=\Spec(k)$ is a point in $U \times_S s$ and $(X_U)_u \cong X_s$.
\end{proof}

Based on the above lemma, the boundedness of a set-theoretic family on $X \rightarrow S$ is equivalent to the boundedness of the corresponding family on $X_U \rightarrow U$, where $X_U$ is a projective scheme over $U$. Therefore, we have the following proposition.
\begin{prop}\label{313}
Let $X$ be a projective algebraic space over $S$ with a polarization $H$ on $X$. Let $\mathfrak{F}$ be a set-theoretic family of coherent sheaves on $X \rightarrow S$. The following statements are equivalent:
\begin{itemize}
\item[(1)] The family $\mathfrak{F}$ is bounded.
\item[(2)] The set of Hilbert polynomials for $F_s \in \mathfrak{F}$ is finite and there is a non-negative integer $m$ such that every $F$ is $m$-regular.
\item[(3)] The set of Hilbert polynomials is finite, and there is a coherent sheaf $\mathcal{G}_T$ on $\mathcal{X}_T$ such that every element $F \in \mathfrak{F}$ admits a surjective morphism $(\mathcal{G}_T)_t \rightarrow F$ for some point $t \in T$.
\end{itemize}
\end{prop}

Now let $\mathcal{X}$ be a projective Deligne-Mumford stack over $S$ with coarse moduli space $X$. A set-theoretic family $\mathfrak{F}$ of coherent sheaves on $\mathcal{X}$ is defined on the fibers of $\mathcal{X} \rightarrow S$.

\begin{defn}\label{314}
A set-theoretic family $\mathfrak{F}$ of coherent sheaves on $\mathcal{X}$ is bounded if there is an $S$-scheme $T$ of finite type and a coherent sheaf $\mathcal{F}_T$ on $\mathcal{X}_T$ such that every sheaf in $\mathfrak{F}$ is contained in the fiber of $\mathcal{F}_T$.
\end{defn}

Let $\mathcal{E}$ be a generating sheaf. Let $F_{\mathcal{E}} : {\rm QCoh}(\mathcal{X}) \rightarrow {\rm QCoh}(X)$ be the exact functor, which is injective. Restricting to a point $s \in S$, the induced functor $F_{\mathcal{E}_s}: {\rm QCoh}(\mathcal{X}_s) \rightarrow {\rm QCoh}(X_s)$ is still exact and injective. Denote by $F_{\mathcal{E}}(\mathfrak{F}):=\{F_{\mathcal{E}_s}(\mathcal{F}_s) \text{ } | \text{ } \mathcal{F}_s \in \mathfrak{F}\}$ the family of coherent sheaves on $X \rightarrow S$.

\begin{prop}\label{315}
The following statements of boundedness are equivalent:
\begin{itemize}
\item[(1)] The set-theoretic family $\mathfrak{F}$ of coherent sheaves on $\mathcal{X}$ is bounded.
\item[(2)] The set of Hilbert polynomials $P_{\mathcal{E}_s}(\mathcal{F}_s)$ for $\mathcal{F}_s \in \mathfrak{F}$ is finite and there is a positive integer $m$ such that $\mathcal{F}_s$ is $m$-regular.
\item[(3)] The set of Hilbert polynomials $P_{\mathcal{E}_s}(\mathcal{F}_s)$ for $\mathcal{F}_s \in \mathfrak{F}$ is finite, and there is a coherent sheaf $\mathcal{G}_T$ on $\mathcal{X}_T$ such that every $\mathcal{F}_s$ is a quotient of $(\mathcal{G}_T)_t$ for some point $t \in T$.
\end{itemize}
\end{prop}

\begin{proof}
This proposition is the ``algebraic space" version of \cite[Theorem 4.12]{Nir}. The setup of this proposition is that $X \rightarrow S$ is a projective morphism of algebraic spaces, while $X \rightarrow S$ is a projective of schemes in Theorem 4.12 in \cite{Nir}. In Lemma \ref{312}, we have already proved that the boundedness of a family of coherent sheaves over algebraic spaces is equivalent to the boundedness of the corresponding family of coherent sheaves over schemes (the \'etale covering). This property implies that we only have to work in the case of schemes. Based on this fact, the proof of this proposition is exactly in the same way as the proof of \cite[Theorem 4.12]{Nir}.
\end{proof}

The above proposition implies the following corollaries.
\begin{cor}\label{316}
A family $\mathfrak{F}$ of coherent sheaves on $\mathcal{X} \rightarrow S$ is bounded if and only if the corresponding family $F_{\mathcal{E}}(\mathfrak{F})$ on $X \rightarrow S$ is bounded.
\end{cor}

\begin{cor}\label{317}
Let $P$ be an integer polynomial. The family of $\mathcal{H}$-semistable sheaves with modified Hilbert polynomial $P$ on $\mathcal{X}$ is bounded.
\end{cor}

\subsection{Construction of the Moduli Space of $\mathcal{H}$-semistable Sheaves}
Now we consider the family $\mathfrak{F}^{\mathcal{H}}_{ss}(P)$ of purely $d$-dimensional $\mathcal{H}$-semistable coherent sheaves with modified Hilbert polynomial $P$. By Corollary \ref{317}, the family $\mathfrak{F}^{\mathcal{H}}_{ss}(P)$ is bounded. Thus we can find an integer $m$ such that $\mathcal{F}$ is $m$-regular for any $\mathcal{F} \in \mathfrak{F}^{\mathcal{H}}_{ss}(P)$. In the classical case (as schemes), there is an upper bound of the set
\begin{align*}
\{h^0(X,F) \text{ } | \text{ } F \text{ is a pure sheaf of dimension $d$ with Hilbert polynomial $P$}\},
\end{align*}
and the upper bound only depends on the maximal slope (note that $F$ may not be semistable), the multiplicity and dimension of $F$ (see \cite[Corollary 3.4]{Lan20041}). F. Nironi generalized this result and proved an upper bound of global sections for the family of $\mathcal{E}$-semistable sheaves on projective Deligne-Mumford stacks, where $\mathcal{E}$ is a generating sheaf (see \cite[Corollary 4.30]{Nir}). This approach also works for any locally free sheaf $\mathcal{H}$ and the family of $\mathcal{H}$-semistable sheaves. In other words, there is an upper bound for the family
\begin{align*}
\{h^0(\mathcal{X},\mathcal{F} \otimes \mathcal{H}^{\vee} \otimes \pi^* \mathcal{O}_X(m) ) \text{ } | \text{ } \mathcal{F} \text{ is a $\mathcal{H}$-semistable sheaf with modified Hilbert polynomial $P$}\},
\end{align*}
and the upper bound only depends on the multiplicity, the dimension $d$ and the slope. Therefore, we can choose a positive integer $N$ large enough such that for any $\mathcal{F} \in \mathfrak{F}^{\mathcal{H}}_{ss}(P)$, we have
\begin{align*}
P(N) \geq P_{\mathcal{H}}(\mathcal{F},m)=h^0(X/S,F_{\mathcal{H}}(\mathcal{F})(m)).
\end{align*}
Denote by $V$ be the linear space $S^{\oplus P(N)}$, and we have
\begin{align*}
V  \cong H^0(X/S,F_{\mathcal{H}}(\mathcal{F})(N)).
\end{align*}
By the above discussion, any coherent sheaf $\mathcal{F} \in \mathfrak{F}^{\mathcal{H}}_{ss}(P)$ corresponds to a surjection
\begin{align*}
V \otimes \mathcal{G} \rightarrow \mathcal{F},
\end{align*}
where $\mathcal{G} \cong \mathcal{H} \otimes \pi^* \mathcal{O}_X(-N)$, together with an isomorphism $V \cong H^0(X/S,F_{\mathcal{H}}(\mathcal{F})(N))$.

Now we consider the algebraic space ${\rm Q}(V \otimes \mathcal{G},\mathcal{X},P)$. Let $[V \otimes \mathcal{G} \rightarrow \mathcal{F}]$ be an element in the algebraic space. Under the exact functor $\pi_*: {\rm Qcoh}(\mathcal{X}) \rightarrow {\rm Qcoh}(X)$, we have a morphism
\begin{align*}
V \otimes \mathcal{O}_X(-N) \rightarrow \pi_*(\mathcal{F} \otimes \mathcal{H}^{\vee}),
\end{align*}
which induces the following one
\begin{align*}
\alpha: V \rightarrow H^0(X/S, (F_{\mathcal{H}}(\mathcal{F}))(N) ).
\end{align*}
Denote by $Q^{\mathcal{H}}$ the subspace of ${\rm Q}^{\mathcal{H}}(V \otimes \mathcal{G},\mathcal{X},P)$ parametrizing quotients $[q: V \otimes \mathcal{G} \rightarrow \mathcal{F}]$ such that
\begin{enumerate}
\item the inducing morphism $\alpha: V \rightarrow H^0(X/S, F_{\mathcal{H}}(\mathcal{F})(N))$ is an isomorphism,
\item $\theta_{\mathcal{H}}( {\rm ker}(q) )$ is surjective.
\end{enumerate}
Both conditions are open condition. Therefore, $Q^{\mathcal{H}} \subseteq {\rm Q}^{\mathcal{H}}(V \otimes \mathcal{G},\mathcal{X},P) \subseteq {\rm Q}(V \otimes \mathcal{G},\mathcal{X},P)$ is an open subset. Denote by $Q_{ss}^{\mathcal{H}}$ the subset of $Q^{\mathcal{H}}$ such that the coherent sheaf $\mathcal{F}$ is $\mathcal{H}$-semistable. The open set $Q^{\mathcal{H}}_{ss} \subseteq {\rm Q}(V \otimes \mathcal{G},P)$ corresponds to the family $\mathfrak{F}^{\mathcal{H}}_{ss}(P)$. With the same approach, we can construct the algebraic space $Q^{\mathcal{H}}_{s} \subseteq {\rm Q}(V \otimes \mathcal{G},P)$ including all $\mathcal{H}$-stable sheaves.

Now we will consider how to construct a GIT quotient of $Q^{\mathcal{H}}$ with respect to the natural $\text{SL}(V)$-action. The functor $F_{\mathcal{H}}$ induces a morphism of quot-spaces
\begin{align*}
{\rm Q}(V \otimes \mathcal{G},\mathcal{X},P) \rightarrow {\rm Q}(F_{\mathcal{H}}(V \otimes \mathcal{G}),X,P).
\end{align*}
Note that this morphism may not be injective. However, it is injective when restricted to $Q^{\mathcal{H}}$ by Corollary \ref{310}. More precisely, $Q^{\mathcal{H}} \hookrightarrow {\rm Q}(F_{\mathcal{H}}(V \otimes \mathcal{G}),X,P)$ is a finitely-presented closed embedding. Also, there is a natural embedding
\begin{align*}
    \psi_N: {\rm Q}( F_{\mathcal{H}}(V \otimes \mathcal{G}), X, P) \hookrightarrow {\rm Grass}(H^0(X/S,  F_{\mathcal{H}}(V \otimes \mathcal{G})(N) ),P(N)
    ),
\end{align*}
where $N$ is a large enough positive integer. Thus, we have
\begin{align*}
Q^{\mathcal{H}} \hookrightarrow {\rm Grass}(H^0(X/S,  F_{\mathcal{H}}(V \otimes \mathcal{G})(N) ),P(N)) = {\rm Grass}(V \otimes H^0(X/S,  F_{\mathcal{H}}( \mathcal{G})(N) ),P(N)).
\end{align*}
Denote by $\mathscr{L}_N$ the pull-back of the canonical invertible sheaf on the Grassmannian via $\psi_N$. Note that the natural group action ${\rm SL}(V)$ on $Q^{\mathcal{H}}$ induces an action on the line bundle $\mathscr{L}_N$. Now we have a group action ${\rm SL}(V)$ on $Q^{\mathcal{H}}$ and an ample line bundle $\mathscr{L}_N$ over $Q^{\mathcal{H}}$. With respect to the line bundle $\mathscr{L}_N$ and the group action $\text{SL}(V)$, we can define the semistable (resp. stable) points on $Q^{\mathcal{H}}$. Next, we will prove that a point $[V \otimes \mathcal{G} \rightarrow \mathcal{F}] \in Q^{\mathcal{H}}$ is semistable if and only if $\mathcal{F}$ is $\mathcal{H}$-semistable (Theorem \ref{320}). Before we prove the statement, we first review two lemmas.

\begin{lem}[Lemma 6.10 in \cite{Nir}]\label{318}
If $\mathcal{F}$ is a coherent sheaf on $\mathcal{X}$ that can be deformed to a pure sheaf of the same dimension $d$, then there is a pure sheaf $\mathcal{K}$ of dimension $d$ on $\mathcal{X}$ and a map $\mathcal{F} \rightarrow \mathcal{K}$ such that the kernel is $T_{d-1}(\mathcal{F})$ and $P_{\mathcal{H}}(\mathcal{F})=P_{\mathcal{H}}(\mathcal{K})$.
\end{lem}

\begin{lem}\label{319}
A point $[V \otimes W \rightarrow U]$ in ${\rm Grass}_S(V \otimes W,a)$, where $a$ is a positive integer, is semistable for the ${\rm SL}(V)$-action and the canonical invertible sheaf if and only if, for all non-trivial proper subspaces $H \subseteq V$, we have ${\rm Im}(H \otimes W) \neq 0$ and
\begin{align*}
\frac{\dim(H)}{\dim({\rm Im}(H \otimes W)  )} \leq \frac{\dim(V)}{\dim(U)}.
\end{align*}
\end{lem}

\begin{proof}
Note that the grassmannian we consider here is over an algebraic space $S$, which means that it is an algebraic space. However, Proposition \ref{222} tells us that the grassmannian in this case has the same property as in the case of schemes (see \cite[Proposition 4.3]{MuFoKir}). Therefore, this lemma is implied by \cite[Proposition 1.14]{Simp2}.
\end{proof}

Based on the above lemmas, we prove the following theorem.

\begin{thm}\label{320}
A point $[V \otimes \mathcal{G} \rightarrow \mathcal{F}] \in Q^{\mathcal{H}}$ is semistable (resp. stable) with respect to the action of $\text{SL}(V)$ and the line bundle $\mathscr{L}_N$, if and only if $\mathcal{F}$ is a $\mathcal{H}$-semistable (resp. $\mathcal{H}$-stable) sheaf of pure dimension $d$ and the map $V \rightarrow H^0(X,F_{\mathcal{H}}(\mathcal{F})(N))$ is an isomorphism.
\end{thm}

\begin{proof}
Take $M$ large enough such that $Q^{\mathcal{H}}$ is embedded into ${\rm Grass}(V \otimes W ,P(M))$, where $W=H^0(X/S,  F_{\mathcal{H}}( \mathcal{G})(M) )$. Let $H$ be a non-trivial proper subspace of $V$ such that the image of $H \otimes W$ is non-empty. Let $\mathcal{F}'$ be the image of $H \otimes W$. We have
\begin{align*}
0 \rightarrow \mathcal{F}'' \rightarrow H \otimes \mathcal{G} \rightarrow \mathcal{F}' \rightarrow 0,
\end{align*}
where $\mathcal{F}''$ is the kernel of the quotient $H \otimes \mathcal{G} \rightarrow \mathcal{F}'$. Since $m$ is a large enough integer, we can assume that
\begin{align*}
h^0(X/S, F_{\mathcal{H}}(\mathcal{F}')(M) )=p_{\mathcal{H}}(\mathcal{F}',M), \quad h^1(X/S,F_{\mathcal{H}}(\mathcal{F}'')(M))=0.
\end{align*}
We have a surjective morphism
\begin{align*}
H \otimes W \rightarrow H^0(X/S, F_{\mathcal{H}}(\mathcal{F}'(M)) ) \rightarrow 0.
\end{align*}
By assumption, we know that $\mathcal{F}$ is $\mathcal{H}$-semistable, which means that
\begin{align*}
\frac{h^0(\mathcal{F}'(N))}{r(\mathcal{F}')} \leq \frac{h^0(\mathcal{F}(N))}{r(\mathcal{F})}.
\end{align*}
Thus, if the integer $M$ is large enough, we have
\begin{align*}
\frac{h^0(\mathcal{F}'(N))}{ P_{\mathcal{H}}(\mathcal{F}',M)  } \leq \frac{h^0(\mathcal{F}(N))}{ P_{\mathcal{H}}(\mathcal{F},M)   },
\end{align*}
and then,
\begin{align*}
\frac{\dim(H)}{\dim({\rm Im}(H \otimes W))}=\frac{h^0(\mathcal{F}'(N))}{ P_{\mathcal{H}}(\mathcal{F}',M)  } \leq \frac{P(N)}{ P(M) }=\frac{\dim(V)}{\dim(U)}.
\end{align*}
This inequality holds for any non-trivial proper subspace $H$ of $V$. By Lemma \ref{319}, the point $[V \otimes \mathcal{G} \rightarrow \mathcal{F}]$ is semistable.

Now we consider another direction. Let $[\rho: V \otimes \mathcal{G} \rightarrow \mathcal{F}]$ be semistable in the sense of GIT. We will prove that $\mathcal{F}$ is a pure $\mathcal{H}$-semistable sheaf and the map $V \rightarrow {\rm H}^0(X,F_{\mathcal{H}}(\mathcal{F})(N))$ is an isomorphism.

We first suppose that $\mathcal{F}$ is pure. Let $\mathcal{F}'$ be a subsheaf of $\mathcal{F}$. By taking the pullback of the following diagram
\begin{center}
\begin{tikzcd}
    V' \arrow[d, dashed, hook] \arrow[r, dashed] & \mathcal{F}' \otimes \mathcal{G}^{\vee} \arrow[d, hook]\\
    V \arrow[r,"\rho"] & \mathcal{F} \otimes \mathcal{G}^{\vee}
\end{tikzcd}
\end{center}
we find a subspace $V' \subseteq V$ such that the quotient $[V' \otimes \mathcal{G} \rightarrow \mathcal{F}']$ is induced by $[\rho]$. Furthermore, $\mathcal{F}$ and $\mathcal{F}'$ have the same regularity and $V' \cong H^0(F_{\mathcal{H}}(\mathcal{F}')(N))$. With the same notation as in the first part of the proof, by taking $N$ and $M$ large enough, we have
\begin{align*}
\frac{h^0(F_{\mathcal{H}}(\mathcal{F}')(N))}{ P_{\mathcal{H}}(\mathcal{F}',M)  }=\frac{\dim(H)}{\dim({\rm Im}(H \otimes W))} \leq \frac{\dim(V)}{\dim(U)}= \frac{P(N)}{ P(M) }.
\end{align*}
This inequality gives us the following
\begin{align*}
\frac{ P_{\mathcal{H}}(\mathcal{F}',N) }{r(\mathcal{F}')} \leq \frac{ P_{\mathcal{H}}(\mathcal{F},N) }{r(\mathcal{F})}, \quad M \gg 0.
\end{align*}
Therefore, $\mathcal{F}$ is $\mathcal{H}$-semistable. Taking $N$ large enough, the induced map
\begin{align*}
V \rightarrow {\rm H}^0(X,F_{\mathcal{H}}(\mathcal{F})(N))
\end{align*}
is surjective. By counting the dimension, this map is an isomorphism. This finishes the proof when $\mathcal{F}$ is pure.

To complete the proof of this theorem, we will show that given any semistable (GIT) point $[\rho]$, the sheaf $\mathcal{F}$ is pure. By Lemma \ref{318}, there exists a pure sheaf $\mathcal{K}$ and a morphism $\varrho: \mathcal{F} \rightarrow \mathcal{K}$ such that
\begin{itemize}
\item the kernel of $\varrho$ is $T_{d-1}(\mathcal{F})$, i.e. the map $\varrho$ is generically injective;
\item $P_{\mathcal{H}}(\mathcal{F})=P_{\mathcal{H}}(\mathcal{K})$.
\end{itemize}
The map $\varrho$ induces an injective map
\begin{align*}
V \xrightarrow{\cong} H^0(X/S, F_{\mathcal{H}}(\mathcal{F})(N)  )\rightarrow H^0(X/S, F_{\mathcal{H}}(\mathcal{K})(N)  ).
\end{align*}
Let $\mathcal{K}''$ be any quotient of $\mathcal{K}$, and denote by $\mathcal{F}'$ the kernel of the composition $\mathcal{F} \rightarrow \mathcal{K} \rightarrow \mathcal{K}''$. We have the following exact sequence
\begin{align*}
0 \rightarrow \mathcal{F}' \rightarrow \mathcal{F} \rightarrow \mathcal{K} \rightarrow \mathcal{K}'' \rightarrow 0.
\end{align*}
This implies
\begin{align*}
h^0(F_{\mathcal{H}}(\mathcal{K}'')(N)) & \geq h^0(F_{\mathcal{H}}(\mathcal{F})(N)) - h^0(F_{\mathcal{H}}(\mathcal{F}')(N))\\
& \geq (r(\mathcal{F})-r(\mathcal{F}')) p_{\mathcal{H}}(\mathcal{F},N) = r(\mathcal{K}'')  p_{\mathcal{H}}(\mathcal{F},N).
\end{align*}
Therefore, $\mathcal{K}$ is $p$-semistable. Furthermore, $V \cong h^0(F_{\mathcal{H}}(\mathcal{K})(N))$. Note that $\varrho$ induces an injection $V \rightarrow H^0(X/S, F_{\mathcal{H}}(\mathcal{K})(N)  )$. By counting the dimension, it is an isomorphism. This isomorphism means that the map $V \otimes \mathcal{G} \rightarrow \mathcal{K}$ factors through $\mathcal{F}$, i.e. the morphism $\varrho: \mathcal{F} \rightarrow \mathcal{K}$ is surjective. Since $P_{\mathcal{H}}(\mathcal{F})=P_{\mathcal{H}}(\mathcal{K})$, we have $\mathcal{F} \cong \mathcal{K}$. This means that $\mathcal{F}$ is pure.
\end{proof}

\begin{lem}\label{321}
Let $[V \otimes \mathcal{G} \rightarrow \mathcal{F}_i]$, $i=1,2$ be two points in $Q_{ss}^{\mathcal{H}}$. The closures of the corresponding orbits in $Q_{ss}^{\mathcal{H}}$ intersect if and only if $gr^{\rm JH}(\mathcal{F}_1) \cong gr^{\rm JH}(\mathcal{F}_2)$.
\end{lem}

\begin{proof}
Let $[\rho: V \otimes \mathcal{G} \rightarrow \mathcal{F}] \in Q_{ss}^{\mathcal{H}}$ be a point. Let
\begin{align*}
0 = {\rm JH}_0(\mathcal{F}) \subseteq {\rm JH}_1(\mathcal{F}) \subseteq \dots \subseteq {\rm JH}_l(\mathcal{F})= \mathcal{F}
\end{align*}
be the Jordan-H\"older filtration of $\mathcal{F}$. To prove the lemma, it is enough to show that we can construct a quotient $[\bar{\rho}: V \otimes \mathcal{G} \rightarrow gr^{\rm JH}(\mathcal{F})]$  such that $[\bar{\rho}]$ is included in the closure of the orbit of $[\rho]$.

Since $N$ is a large enough integer, we can assume that $F_{\mathcal{H}}({\rm JH}_i(\mathcal{F}))(N)$ is globally generated, and let $V_{\leq i}$ be the subspace of $V$ such that the quotient $[V_{\leq i} \otimes \mathcal{G} \rightarrow {\rm JH}_i(\mathcal{F})]$ is induced by $[\rho]$ and $V_{\leq i} \cong H^0(X/S, F_{\mathcal{H}}({\rm JH}_i(\mathcal{F}))(N))$. Let $V_i: = V_{\leq i}/ V_{\leq i-1}$. We have the induced surjections $V_i \otimes \mathcal{G} \rightarrow gr_i^{\rm JH}(\mathcal{F})$. Summing up these induced surjections, we get a point $[\bar{\rho}: V \otimes \mathcal{G} \rightarrow gr^{\rm JH}(\mathcal{F})]$.

To show that $[\bar{\rho}]$ is in the closure of the orbit of $[\rho]$, it suffices to find an one-parameter subgroup $\lambda$ such that $\lim\limits_{t \rightarrow 0} \lambda(t) \cdot [\rho]=[\bar{\rho}]$. The construction of such an one-parameter subgroup $\lambda$ is the same as \cite[Lemma 4.4.3]{HuLe}. Therefore, the point $[\bar{\rho}]$ is included in the closure of the orbit of $[\rho]$ in $Q^{ss}$.
\end{proof}

As we discussed above, a point $[V \otimes \mathcal{G} \rightarrow \mathcal{F}] \in Q^{\mathcal{H}}$ is $\mathcal{H}$-semistable if and only if it is semistable in the sense of GIT. Therefore, a GIT quotient exists for $Q^{\mathcal{H}}$, and the semistable locus is exactly $Q_{ss}^{\mathcal{H}}$ by Theorem \ref{217} and Theorem \ref{320}. Denote by
\begin{align*}
\mathcal{M}^{ss}(\mathcal{H},\mathcal{O}_X(1),P):=Q^{\mathcal{H}}_{ss}/{\rm SL}(V)
\end{align*}
the universal good quotient with respect to the group action ${\rm SL}(V)$ and line bundle $\mathscr{L}_N$.

\begin{thm}\label{322}
$\mathcal{M}^{ss}(\mathcal{H},\mathcal{O}_X(1),P)$ is the coarse moduli space for the $S$-equivalence classes of $\mathcal{H}$-semistable coherent sheaves with modified Hilbert polynomial $P$, and $\mathcal{M}^{ss}(\mathcal{H},\mathcal{O}_X(1),P)$ is an algebraic space projective over $S$.
\end{thm}

\bigskip
\noindent\small{\textsc{Department of Mathematics, South China University of Technology}\\
381 Wushan Rd, Tianhe Qu, Guangzhou, Guangdong, China}\\
\emph{E-mail address}:  \texttt{hsun71275@scut.edu.cn}

\end{document}